\documentclass[leqno,11pt]{amsart}
\usepackage{amssymb, amsmath,amsmath,latexsym,amssymb,amsfonts,amsbsy, amsthm,esint}
\usepackage{color}
\usepackage{graphicx}
\usepackage{tikz}
\usetikzlibrary{arrows, calc, decorations.markings,intersections, patterns,patterns.meta}
\usepackage{caption}
\usepackage{subcaption}
\usepackage{hyperref}

\usepackage{pgfplots}
\pgfplotsset{compat=1.15}
\usepackage{mathrsfs}
\usepackage[foot]{amsaddr}

\makeatletter
\DeclareFontFamily{U}{tipa}{}
\DeclareFontShape{U}{tipa}{m}{n}{<->tipa10}{}
\newcommand{\arc@char}{{\usefont{U}{tipa}{m}{n}\symbol{62}}}%

\newcommand{\arc}[1]{\mathpalette\arc@arc{#1}}

\newcommand{\arc@arc}[2]{%
  \sbox0{$\m@th#1#2$}%
  \vbox{
    \hbox{\resizebox{\wd0}{\height}{\arc@char}}
    \nointerlineskip
    \box0
  }%
}
\makeatother

\let\pa=\partial
\let\al=\alpha

\let\d=\delta

\let\lam=\lambda

\let\f=\frac

\let\G= \Gamma

\let\Om=\Omega

\let\e=\varepsilon
\let\pa=\partial

\let\ri=\rightarrow
\let\na=\nabla

\newcommand{\beq}{\begin{equation}}
\newcommand{\eeq}{\end{equation}}
\newcommand{\beqo}{\begin{equation*}}
\newcommand{\eeqo}{\end{equation*}}
\newcommand{\ben}{\begin{eqnarray}}
\newcommand{\een}{\end{eqnarray}}
\newcommand{\beno}{\begin{eqnarray*}}
\newcommand{\eeno}{\end{eqnarray*}}

\newtheorem{theorem}{Theorem}[section]

\newtheorem{lemma}[theorem]{Lemma}
\newtheorem{proposition}[theorem]{Proposition}
\newtheorem{corol}[theorem]{Corollary}

\theoremstyle{remark}

\newtheorem{rmk}{Remark}[section]

\newcommand{\dist}{\mathrm{dist}}

\newcommand{\BR}{\mathbb{R}}

\newcommand{\ch}{\mathcal{H}^1}

\newcommand{\mres}{\mathbin{\vrule height 1.6ex depth 0pt width 0.13ex\vrule height 0.13ex depth 0pt width 1.1ex}}

\begin{document}

\title[Rigidity of triple junction]{Rigidity results for a triple junction solution of Allen-Cahn system}

\author{Zhiyuan Geng}
\address{Department of Mathematics, Purdue University, 150 N. University Street, West Lafayette, IN 47907–2067}
\email{geng42@purdue.edu}

\begin{abstract} For the two dimensional Allen-Cahn system with a triple-well potential, previous results established the existence of a minimizing solution $u:\mathbb{R}^2\rightarrow\mathbb{R}^2$ with a triple junction structure at infinity. We show that along each of three sharp interfaces, $u$ is asymptotically invariant in the direction of the interface and can be well-approximated by the 1D heteroclinic connections between two phases. Consequently, the diffuse interface is located in an $O(1)$ neighborhood of the sharp interface, and becomes nearly flat at infinity. This generalizes all the results for the triple junction solution with symmetry hypotheses to the non-symmetric case. The proof relies on refined sharp energy lower and upper bounds, alongside a precise estimate of the diffuse interface location. 
\end{abstract}

\keywords{Allen-Cahn system, diffuse interface, triple junction solution, heteroclinic connection, De Giorgi conjecture}

\maketitle

\section{Introduction}

In this paper, we investigate the bounded entire solution of the system 
\begin{equation}\label{eq: AC system}
    \Delta u-W_u(u)=0,\quad u:\BR^2\ri \BR^2,
\end{equation}
which is minimizing on compact sets the associated Allen-Cahn energy:
\begin{equation*}
    E(u,\Omega):=\int_{\Omega} \left(\f12|\na u|^2+W(u)\right)\,dx.
\end{equation*}

Specifically for $W$ we assume
\vspace{3mm}
\begin{enumerate}
    \item[(H1).]  $W\in C^2(\BR^2;[0,+\infty))$, $\{z: \,W(z)=0\}=\{a_1,a_2,a_3\}$, $W_u(u)\cdot u >0$ if $|u|>M$ and 
    \beqo
     c_2|\xi^2| \geq \xi^TW_{uu}(a_i)\xi\geq  c_1|\xi|^2,\; i=1,2,3.
    \eeqo 
    for some positive constants $c_1<c_2$ depending on $W$. 
\item[(H2).] For $i\neq j$, $i,j\in \{1,2,3\}$, there exists a unique (up to translation) minimizing heteroclinic connection $U_{ij}\in W^{1,2}(\BR,\BR^2)$ that minimizes the one dimensional energy functional
\beqo
J(U):=\int_{\BR}\left(\f12|U'|^2+W(U)\right)\,d\eta, \quad \lim\limits_{\eta\ri-\infty}U(\eta)=a_i,\ \lim\limits_{\eta\ri+\infty}U(\eta)=a_j.
\eeqo
The connection $U_{ij}$ is non-degenerate in the sense that $0$ is a simple eigenvalue of the operator $\mathcal{T}: W^{2,2}(\BR,\BR^2)\ri L^2(\BR,\BR^2)$:
\begin{equation}\label{def: 1D EL operator}
    \mathcal{T}\varphi:=-\varphi''+W_{uu}(U_{ij})\varphi.
\end{equation}
Let $\sigma_{ij}$ denote the minimal energy $J(U_{ij})$. Assume
\begin{equation}\label{all equal sigma}
    \sigma_{ij}\equiv \sigma,\quad \forall i\neq j\in\{1,2,3\}. 
\end{equation}
\end{enumerate}

The minimizing solution we seek for is a diffuse analogue of the planar minimizing partition. Let $\Omega$ (could be the whole $\BR^2$) denote a two dimensional domain.  A partition $\mathcal{P}=\{P_1,P_2,P_3\}$ of $\Omega$ is called a locally minimizing partition of $\Omega$, if for any compact set $K\subset \Omega$,
\begin{equation*}
E_0(\{P_1,P_2,P_3\},K)\leq E_0(\{A_1,A_2,A_3\},K),
\end{equation*}
where $E_0(\cdot)$ denotes the functional
\begin{equation}\label{functional: min par}
        E_0(\{D_1,D_2,D_3\},K):=\sum\limits_{1\leq i<j\leq 3} \sigma \ch(\pa^* D_i\cap \pa^* D_j\cap K),
\end{equation}
and $\mathcal{A}=\{A_i\}_{i=1}^3$ is any 3--partition of $\Om$ such that  $\mathcal{P}\mres(\Om\setminus K)= \mathcal{A}\mres(\Om\setminus K).$ Here $\pa^*$ refers to the reduced boundary of a set of finite perimeter. Analogously, for any function $v\in \mathrm{BV}_{loc}(\BR^2,\{a_1,a_2,a_3\})$, we can define the energy functional
\begin{equation}\label{functional: min par map}
    E^*_0(v,K):= E_0(\{v^{-1}(a_i)\}_{i=1}^3, K).
\end{equation}
$v$ is called a minimizing partition map on $\Omega$ if $v$ minimizes $E_0^*(\cdot, K)$ for any $K\Subset \Om$. 

In particular, when $\Omega=\BR^2$, we consider the following minimizing 3-partition of $\BR^2$,
\beq\label{partition p}
\begin{split}
&\qquad\qquad\qquad \mathcal{P}:=\{\mathcal{D}_1,\mathcal{D}_2,\mathcal{D}_3\}, \\
&\mathcal{D}_i:=\{(r\cos\theta,r\sin\theta) : r\in(0,\infty), \theta\in (\f{2(i-1)\pi}{3},\f{2i\pi}{3})\},\ i=1,2,3,
\end{split}
\eeq
which is a partition of the plane into three sectors centered at the origin with opening angles of $\f{2\pi}{3}$. The sharp interface that separates these sub-domains is denoted by 
\begin{equation*}
  \pa\mathcal{P}:=\pa \mathcal{D}_1\cup\pa \mathcal{D}_2\cup\pa \mathcal{D}_3.  
\end{equation*}

The \emph{triple junction map} is defined by 
\beq\label{form of triple junction sol}
u_{\mathcal{P}}:= a_1\chi_{\mathcal{D}_1}+ a_2\chi_{\mathcal{D}_2}+a_3\chi_{\mathcal{D}_3},
\eeq
where $\chi_{\Om}$ represents the characteristic function of domain $\Om$. The minimality of $\mathcal{P}$ and $u_{\mathcal{P}}$ defined above is related to the Steiner point of triangle in classical Euclidean geometry. 
 


    

Here is our main result.

\begin{theorem}\label{thm: rigidity triple junction} Let $u: \BR^2\ri \BR^2$ be a minimizing entire solution of \eqref{eq: AC system} with a triple junction structure at infinity, i.e. 
\begin{equation}\label{L1 conv condition main thm}
u(Rz)\xrightarrow[]{L_{loc}^1} u_{\mathcal{P}}\ \text{ as }R\ri\infty.
\end{equation}
For any $i,j \in \{1,2,3\}$, let $\mathbf{e}_{ij}$ denote the unit vector representing the direction of $\pa \mathcal{D}_i\cap\pa\mathcal{D}_j$ and $\mathbf{e}_{ij}^{\perp}$ as its orthogonal unit vector. Then there exists a constant $h_{ij}$, such that 
\begin{equation}\label{main result: conv}
\lim\limits_{x\ri+\infty}\|u(x\mathbf{e}_{ij}+y\mathbf{e}_{ij}^{\perp})-U_{ij}(y-h_{ij})\|_{C^{2,\alpha}(\BR;\BR^2)}=0, \quad \forall \al\in(0,1).
\end{equation}
Furthermore, the following  pointwise estimate holds:
\begin{equation}\label{main result: decay}
    |u(z)-a_i|\leq Ce^{-k\,\dist(z,\pa\mathcal{P})}, \quad \text{for }z\in \mathcal{D}_i,
\end{equation}
where the positive constants $C$, $k$ depend only on $W$ and $u$.
\end{theorem}

\begin{rmk}
    We consider the case where all $\sigma_{ij}$ are equal for simplicity. The result can be extended to the general $\sigma_{ij}$ case without significant technical difficulty. 
\end{rmk}

\begin{rmk}
    By \cite[Proposition 4.3]{ss2024} and \cite[Theorem 1.2]{geng2024uniqueness}, the condition \eqref{L1 conv condition main thm} can be relaxed to the following: there exists a point $z\in \BR^2$ such that 
    \begin{equation*}
        \dist(u(z),\Lambda)>0,
    \end{equation*}
    where $\Lambda=\overline{U_{12}(\BR)\cup U_{23}(\BR)\cup U_{31}(\BR)}$ denotes the closure of three heteroclinic connections.  This remarkably weak condition is sufficient to establish \eqref{L1 conv condition main thm}.  
\end{rmk}

In the scalar case of \eqref{eq: AC system}, where $u: \BR^N\rightarrow \BR$ and $W(u)$ has only two energy wells $u=\pm 1$, the solutions are closely connected to the minimal surface theory, leading to the famous conjecture of De Giorgi \cite{Degiorgi1978}. The conjecture proposes that in dimension $N\leq 8$, any entire solution (not necessarily minimizing) that is monotonic in one variable only depends on that variable, meaning that all level sets are parallel hyperplanes orthogonal to the direction of monotonicity. There are many deep results contributing to the understanding of De Giorgi's conjecture and the relationship between Allen-Cahn equation and minimal surfaces, see for example \cite{ambrosio2000entire,farina20111d,ghoussoub1998conjecture,savin2009regularity,wang2017new,del2011giorgi,del2013entire,guaraco2018min,liu2017global,pacard2013stable} and the expository papers \cite{savin2010minimal,chan2018giorgi} for a detailed account.

In the vector-valued case, minimizing solutions are related to minimal partitions. The convergence of the vector-valued Allen-Cahn system to a minimal partition problem can be established by the $\Gamma$-convergence technique, see \cite{Baldo,sternberg1988effect,fonseca1989gradient,gazoulis} for further details. For the two dimensional case, recent developments have provided significant insights into the geometric and analytic description of fine structures of the minimizing Allen-Cahn solutions. Bethuel \cite{bethuel2021asymptotics} discovered a new monotonicity formula and successfully generalized many results on the regularity of interfaces from the scalar case to the vectorial case. Fusco \cite{fusco2024connectivity} studied the connectivity of the diffuse interface and showed the existence of a connected network with a well-defined structure that separates all phases. Alikakos and Fusco \cite{AF} investigated two examples with carefully designed Dirichlet boundary data, where a sharp lower bound for the energy minimizer could be derived, which in turn yielded precise pointwise estimates. 

As for the triple junction solution on $\mathbb{R}^2$, the first existence result was due to Bronsard, Gui and Schatzman \cite{bronsard1996three}, where an entire solution to \eqref{eq: AC system} was constructed in the equivariant class of the reflection group $\mathcal{G}$ corresponding to the symmetries of the equilateral triangle. The triple-well potential is also assumed invariant under $\mathcal{G}$. The solution is obtained as a minimizer in the equivariant class $u(gx)=gu(x), g\in \mathcal{G}$, hence is not necessarily stable under general perturbations. The result was later extended to the three dimensional case by Gui and Schatzman \cite{gui2008symmetric}. More recently, Fusco \cite{fusco} established the result of \cite{bronsard1996three} in the equivariant class of the rotation subgroup of $\mathcal{G}$ only, thus eliminating the two reflections. For a more detailed discussion, we refer to the book \cite{afs-book} and the references therein.

On bounded domains, several constructions of triple junction solutions without symmetry assumptions have been established. For instance, Sternberg and Ziemer \cite{sternberg1994local} constructed solutions on clover-shaped domains in $\BR^2$ using $\Gamma$-convergence; while Flores, Padilla, and Tonegawa \cite{flores2001higher} extended the construction to more general domains via a mountain pass argument. Regarding entire solutions without symmetry assumptions, Schatzman \cite{schatzman2002asymmetric} demonstrated the existence of a 2D solution connecting two heteroclinic connections for a double-well potential. Her result was later revisited in \cite{fusco2017layered,monteil2017metric,smyrnelis2020connecting}.

Recently, Alikakos and the author \cite{alikakos2024triple}, and Sandier and Sternberg \cite{ss2024}, independently established the existence of an entire minimizing solution, characterized by a triple junction structure at infinity without imposing symmetry assumptions. Using distinct methods, both studies obtained comparable results saying that along a subsequence $R_k\ri\infty$, the rescaled function $u_{R_k}(z):=u(R_kz)$ converges in $L^1_{loc}(\BR^2)$ to a triple junction map $u_{\mathcal{P}}$ of the form \eqref{form of triple junction sol}. Following these two works, the author \cite{geng2024uniqueness} showed the uniqueness of the blow-down limit $u_{\mathcal{P}}$, thus establishing that $u(Rz)\ri u_{\mathcal{P}}$ in $L^1_{loc}$ as $R\ri\infty$. These results will be discussed in more detail later in Section \ref{preliminary}.

A noteworthy byproduct of \cite{ss2024} is the first vectorial analogue of the De Giorgi's conjecture for 2D minimizing Allen-Cahn solutions of two phases. It states that if the blow-down limit of $u$ consists of only two phases $a_1$ and $a_2$, separated by a straight line, then $u$ must be invariant along the direction of this line. This result follows from sharp energy lower and upper bounds, which force the directional derivative of $u$ along the interface to be negligible, thereby ensuring the flatness of the interface. However, the same argument does not apply to the triple junction solution due to the lack of a sharp lower bound. The current paper addresses this challenge. Theorem \ref{thm: rigidity triple junction} establishes that the diffuse interface defined in \eqref{def: diffuse interface} forms a strip of width $O(1)$, which is asymptotically flat at infinity, and that $u$ is nearly an one-dimensional solution along the interface. This result generalizes the complete results of \cite{bronsard1996three} to the non-symmetric setting, as well as confirms a De Giorgi-type conjecture for Allen Cahn solutions with triple phases.

We now outline some key steps in the proof of Theorem \ref{thm: rigidity triple junction}. We start with the minimizing entire solution constructed in \cite{ss2024} and \cite{alikakos2024triple}, which satisfies condition \eqref{L1 conv condition main thm} for the partition $\mathcal{P}$ defined in \eqref{partition p}. Moreover, as shown in \cite{geng2024uniqueness}, the intersection of the diffuse interface $\G_\d:=\{z:\min\limits_{i=1,2,3}|u(z)-a_i|\geq \delta\}$ with $B_R$ is contained in an $O(R^\beta)$ neighborhood of the sharp interface $\pa\mathcal{P}$, for some $\beta\in (\f12,1)$. Outside this $O(R^\beta)$ neighborhood, the distance of $u(z)$ to $a_i$ is controlled by the estimate
\begin{equation*}
    |u(z)-a_i|\leq Ce^{-k(\dist(z,\pa \mathcal{P})- CR^{\beta})},\quad z\in B_R\cap D_i\cap \{dist(z, \pa\mathcal{P})>CR^{\beta}\}.
\end{equation*}

This exponential decay implies that on any large circle $\pa B_R$, the restriction $u|_{\pa B_R}$ approaches three phases respectively. Thus there must be at least three phase transitions along $\pa B_R$, which contribute a minimal energy of $3\sigma$ in the tangential direction. Integrating with the radius $R$ yields the sharp energy bound
\begin{equation}\label{intro: sharp bdd}
     3\sigma R-C\leq \int_{B_R}\left( \f12|\na u|^2 +W(u)
 \right)\,dz\leq 3\sigma R+C,
\end{equation}
as established in Lemma \ref{lemma: sharp bdd on BR}. From this energy bound, one can update the power $\beta$, which measures the closeness of $\G_\d$ to $\pa\mathcal{P}$, to $\f12$.

Since the tangential deformation and the potential take most of the energy, the radial deformation is relatively small: 
\begin{equation*}
    \int_{\BR^2}|\pa_r u|^2\,dxdy<\infty.
\end{equation*}
Assume the $a_1$-$a_3$ interface lies along the positive $x$-axis. Through technical arguments, it follows that
\begin{equation}\label{intro: pa x small}
    \int_{\{x>0\}}|\pa_x u|^2\,dxdy< \infty.
\end{equation}

For most $x>0$, the vertical slice $l_x:=\{(x,y): y\in \BR \}$ connects $a_3$ to $a_1$ from $(x,-\infty)$ to $(x,\infty)$, with an energy close to $\sigma$. By \cite{schatzman2002asymmetric}, this implies that $u|_{l_x}$ is close to the heteroclinic connection $U_{31}$ in the $H^1$ norm. Let $h(x)$ denotes the optimal translation of $U_{31}$ such that the $L^2$ distance between $u(x,\cdot)$ and $U_{31}(\cdot-h(x))$ is minimized. As shown in \cite{schatzman2002asymmetric}, $h(x)$ is a $C^2$ function of $x$ provided $u(x,\cdot)$ stays sufficiently close to $U_{31}$. A direct calculation yields
\begin{equation*}
    |h'(x)|\sim C\int_{-\infty}^\infty |\pa_x u(x,y)|^2\,dy,
\end{equation*}
which together with \eqref{intro: pa x small} implies that 
$h(x)$ converges to some finite value $h_{31}$ as $x\ri\infty$. Consequently, $u(x,\cdot)$ converges to $U_{31}(\cdot-h_{31})$ in $L^2$. This $L^2$ convergence can be upgraded to $C^{2,\alpha}$ convergence by standard Schauder estimates, thus proving \eqref{main result: conv}. Furthermore, this result indicates that the diffuse interface locates within an $O(1)$ neighborhood of $\pa \mathcal{P}$. From this, the pointwise estimate \eqref{main result: decay} follows from the comparison principle in elliptic theory. 

The article is organized as follows. In Section \ref{preliminary}, we review fundamental estimates for the minimizing solution of \eqref{eq: AC system} from \cite{AF,afs-book} and summarize results in \cite{ss2024,alikakos2024triple,geng2024uniqueness} on the existence of an entire triple junction solution. The proof of Theorem \ref{thm: rigidity triple junction} is presented in Section \ref{sec: proof}. We begin by proving \eqref{intro: sharp bdd} and refining the localization of the diffuse interface, then derive the horizontal deformation estimate \eqref{intro: pa x small}, and finally conclude the proof by analyzing the optimal translation function $h(x)$.

\section{Preliminaries}\label{preliminary}

Throughout the paper we denote by $z=(x,y)$ a two dimensional point and by $B_r(z)$ the two dimensional disk centered at the point $z$ with radius $r$. In addition, we let $B_r$ denote the disk centered at the origin. Moreover, without specific explanation, $C$ denotes a constant that depends on the potential $W$ and on the solution $u$. $C$ may have distinct values in various estimates. We first recall the following basic results (without proofs) which play a crucial part in our analysis.
\begin{lemma}[Lemma 2.1 in \cite{AF}]\label{lemma: potential energy estimate}
The hypotheses on $W$ imply the existence of $\delta_W>0$, and constants $c,C>0$ such that
\beqo
\begin{split}
&|u-a_i|=\delta\\
\Rightarrow & \ \f12 c\delta^2\leq W(u)\leq \f12 C \delta^2,\quad \forall \delta<\delta_W,\ i=1,2,3.
\end{split}
\eeqo
Moreover if $\min\limits_{1\leq i\leq N} |u-a_i|\geq \delta$ for some $\delta<\delta_W$, then $W(u)\geq \f12 c\delta^2.$
\end{lemma}

\begin{lemma}[Lemma 2.3 in \cite{AF}]\label{lemma: 1D energy estimate}
Take $i\neq j \in \{1,...,N\}$, $\d <\d_W$ and $s_+>s_-$ be two real numbers. Let $v:(s_-,s_+)\ri \BR^2$ be a smooth map that minimizes the energy functional 
\beqo
J_{(s_-,s_+)}(v):=\int_{s_-}^{s_+} \left(\f12|\na v|^2+W(v)\right)\,dx 
\eeqo
subject to the boundary condition 
\beqo
|v(s_-)-a_i|=|v(s_+)-a_j|=\delta.
\eeqo
Then
\beqo
J_{(s_-,s_+)}(v)\geq \sigma_{ij}-C\delta^2,
\eeqo
where $C$ is the constant in Lemma \ref{lemma: potential energy estimate}. 
\end{lemma}

\begin{lemma}[Variational maximum principle, \cite{AF2}]\label{lemma: maximum principle}
     There exists a positive constant $r_0=r_0(W)$ such that for any $u\in W^{1,2}(\Omega,\BR^2)\cap L^\infty(\Om,\BR^2)$ being a minimizer of $E(\cdot,\Om)$, if $u$ satisfies 
    $$
    \vert u(x)-a_i\vert\leq r \text{ on }\pa\Om, \ \ \text{for some }r<r_0,\ i\in\{1,2,3\},
    $$
    then 
    $$
    \vert u(x)-a_i\vert\leq r\ \ \forall x\in \Om.
    $$
\end{lemma}

Suppose $u$ is a bounded minimizing solution of \eqref{eq: AC system}. Schauder estimates imply that there exists a constant $M$ such that
\begin{equation}\label{reg of u}
    \|u\|_{C^{2,\alpha}(\BR^2,\BR^2)}\leq M,\quad \forall \al\in(0,1).
\end{equation}
To characterize the blowdown limits of $u$, we invoke the following compactness result.
\begin{proposition}[Proposition 3.1 in \cite{ss2024}] \label{prop: conv of blowdown} Let $u:\BR^2\rightarrow \BR^2$ be a minimizing solution of \eqref{eq: AC system} and $\{r_j\}\ri \infty$ be any sequence. Then there exists a subsequence $\{r_{j_k}\}$ and a function $u_0\in \mathrm{BV}_{loc}(\BR^2,\{a_i\}_{i=1}^N)$ such that the blowdowns $\{u_{r_{j_k}}(z)=u(r_{j_k}z)\}$ satisfy
\beqo
u_{r_{j_k}}\rightarrow u_0 \quad \text{in }L_{loc}^1(\BR^2,\BR^2).
\eeqo
Here $u_0$ is a minimizing partition map on $\BR^2$ in the sense that 
\beqo
E_0^*(u_0,K)\leq E_0^*(v,K)
\eeqo
for every compact set $K$ and every $v\in \mathrm{BV}_{loc}(\BR^2, \{a_i\}_{i=1}^3)$ such that $v=u_0$ on $\BR^2\setminus K$. Along the same subsequence, the following energy estimate holds
\beq\label{ene conv: general}
E_{r_{i_k}}(u_{r_{j_k}},K)\rightarrow E_0^*(u_0,K), \quad \forall K\Subset \BR^2,
\eeq
where $E_R$ denotes the rescaled energy
\beqo
E_R(v,\Om):=\int_{\Omega} \left(\f1{2R} |\na u|^2+RW(u)\right)\,dz.
\eeqo
\end{proposition}

Note that the energy convergence result \eqref{ene conv: general} is not explicitly stated in \cite[Proposition 3.1]{ss2024}, but can be derived from the $\G$--convergence result in Baldo \cite{Baldo} that holds also without the mass constraint (see Gazoulis \cite{gazoulis}). 

By a ``clearing-out" argument, the $L^1_{loc}$ convergence can be improved to uniform convergence outside the support of $\na u_0$. 
\begin{proposition}[Proposition 4.2 in \cite{ss2024}]\label{prop: local uniform convergence}
Let $\{u_{r_{j}}\}$ be a sequence of blowdowns of the minimizing solution $u$ that converges in $L^1_{loc}$ to $u_0\in \mathrm{BV}_{loc}(\BR^2, \{a_i\}_1^N)$. Then $\{u_{r_{j}}\}$ converges to $u_0$ uniformly outside the support of $\na u_0$.  
\end{proposition}
The uniform convergence result above can also be derived from the the vector version of the Caffarelli--C\'{o}rdoba density estimate \cite{AF3}, see for example \cite[Proposition 5.3]{afs-book}.

Moreover, utilizing a Pohozaev identity and estimates of energy lower and upper bounds, \cite{ss2024} establishes an asympototic monotonicity formula and an energy equipartition result, which further implies the homogeneity of the blowdown limit $u_0$. Related results are summarized below. 

\begin{proposition}[Lemma 3,4, Lemma 3.5 \& Theorem 3.6 in \cite{ss2024}]\label{prop: homogeneity of u0}
    Let $u:\BR^2\ri\BR^2$ be a minimizing solution of \eqref{eq: AC system} and $u_0$ be a blowdown limit guaranteed by Proposition \ref{prop: conv of blowdown}. Then $u_0$ is a homogeneous map, i.e.
    \beqo
    u_0(z)=u_0(\frac{z}{|z|}),\quad \forall z\in\BR^2. 
    \eeqo
    The following energy limits hold:
    \begin{equation}\label{eq: energy limit}
        \lim\limits_{R\ri\infty} \f{1}{R} \int_{B_R}W(u)\,dz=\lim\limits_{R\ri\infty} \f1R\int_{B_R} \f12|\na u|^2\,dz=\lim\limits_{R\ri\infty} \f1R\int_{B_R} \f12|\pa_T u|^2\,dz=\frac{1}{2} E_0^*(u_0,B_1), 
    \end{equation}
    where $\pa_T$ means the tangential derivative on $\pa B_R$. Moreover, for every positive $\lam_1<\lam_2$ it holds that 
    \beqo
    \begin{split}
        &\qquad \lim_{R\ri\infty} \f{1}{R}\int_{B_{\lam_2 R}\setminus B_{\lam_1 R}} |\pa_r u|^2\,dz=0,\\
        & \lim_{R\ri\infty} \f{1}{R}\int_{B_{\lam_2 R}\setminus B_{\lam_1 R}} \bigg|W(u)-\f12|\na u|^2\bigg|\,dz=0,
    \end{split}
    \eeqo
where $\pa_r u=\na u \cdot \f{x}{|x|}$ represents the radial derivative of $u$.
\end{proposition}

Proposition \ref{prop: conv of blowdown} and \ref{prop: homogeneity of u0} together imply the following theorem. 
\begin{theorem}[Theorem 1.1 in \cite{ss2024}]\label{thm: triple junction existence}
    There exists a minimizing entire solution $u:\BR^2\ri\BR^2$ to \eqref{eq: AC system} such that for any compact set $K\subset \BR^2$, 
    \begin{equation}\label{L1 conv with rotation}
        \lim\limits_{R\ri\infty} \left(\inf\limits_{\theta}\|u_R(z)-u_{\mathcal{P}}(G_\theta z)\|_{L^1(K;\BR^2)}\right)
    \end{equation}
    where $u_R=u(Rz)$ is the blowdown of $u$, $u_{\mathcal{P}}$ is the triple junction map defined in \eqref{form of triple junction sol} and $G_\theta$ denotes the rotation through an angle $\theta$ about the origin. 
\end{theorem}

Equivalently, this theorem tells that along any sequence $R_i\ri\infty$, there is a subsequence, still denoted by $\{R_i\}$, such that $u_{R_i}$ converges in $L^1_{loc}$ to a triple junction map which may depend on the sequence. A comparable result has been obtained in \cite{alikakos2024triple} by N. Alikakos and the author around the same time with a different method. 

Finally, we mention the following uniqueness result of the blow-down limit which is proved by the author in \cite{geng2024uniqueness}, thereby eliminating the possible rotation $G_\theta$ in \eqref{L1 conv with rotation}.

\begin{theorem}\label{thm: uniqueness}
    For the entire solution $u$ in Theorem \ref{thm: triple junction existence}, there exists a unique triple junction map $u_{\mathcal{P}}$, such that
    \begin{equation}\label{L1 conv}
        \lim\limits_{R\ri\infty}\|u_R-u_{\mathcal{P}}\|_{L^1_{loc}(\BR^2)}=0.
    \end{equation}
\end{theorem}

\section{Proof of Theorem \ref{thm: rigidity triple junction}}\label{sec: proof}

Throughout this section, $u: \BR^2\ri\BR^2$ denotes a minimizing  solution of \eqref{eq: AC system} satisfying Theorem \ref{thm: triple junction existence} and \ref{thm: uniqueness}. The minimizing partition $\mathcal{P}$ and its associated triple junction map $u_{\mathcal{P}}$ are given by \eqref{partition p} and \eqref{form of triple junction sol} respectively. To prove Theorem \ref{thm: rigidity triple junction}, we focus on the interface $\pa \mathcal{D}_1\cap\pa \mathcal{D}_3=\{(x,0):x>0\}$. The proof is identical for the other two interfaces. In this setting, we pick
\begin{equation*}
        \mathbf{e}_{13}:=(1,0) \text{ is the directional unit vector of the interface},\quad\mathbf{e}_{13}^\perp:=(0,1). 
\end{equation*}
By the hypothesis (H2), there exists a unique (up to a translation) minimizing heteroclinic connection $U_{31}$,  such that $U_{31}(-\infty)=a_3,\, U_{31}(+\infty)=a_1$. For convenience, in the rest of the paper we drop the subscript and simply write
\begin{equation*}
    U=U_{31}.
\end{equation*}

Define the diffuse interface of $u$ as
\begin{equation}\label{def: diffuse interface}
    \Gamma_\delta:= \{z\in \BR^2: \min\limits_{i=1,2,3}|u(z)-a_i|\geq \d\},\ \  \delta>0.
\end{equation}

\subsection{Localization of the diffuse interface}

From Theorem \ref{thm: uniqueness}, we know $u$ can be approximated by $u_{\mathcal{P}}$ at large scales. Next we invoke a more quantified version of this approximation from \cite{geng2024uniqueness}, which states that the diffuse interface is located in a small neighborhood of $\pa\mathcal{P}$. 

\begin{proposition}\label{prop: appr uR}
There exist $R_0$, $C$, $\alpha\in(0,1)$, $K$, $k$ depending on $W$ and $u$, such that for every $R>R_0$, the following hold:
\begin{enumerate}
    \item There are points $O_R\in B_R$, $D^1_R,D^2_R,D^3_R\in\pa B_R$ such that the line segments $O_RD^1_R$, $O_RD^2_R$, $O_RD^3_R$ form a $\f{2\pi}{3}$ angle pairwisely, and the following energy estimate holds:
    \begin{equation}\label{energy bdd on BR}
        \left|E(u, B_R)-\sigma \mathcal{H}^1(T_R)\right|\leq CR^\alpha,
    \end{equation}
    where $T_R$ denotes $O_RD^1_R\cup O_RD^2_R\cup O_RD^3_R$. 
    \item The diffuse interface locates in an $O(R^{\f{\al+1}{2}})$ neighborhood of $T_R$, i.e.
    \begin{equation}\label{diff interface loc}
        \Gamma_\delta\cap B_R \subset \{z\in B_R: \dist(z,T_R)\leq C R^{\f{\al+1}{2}}\}.
    \end{equation}
    Moreover, 
    \begin{equation}\label{exp decay: TR}
        \min\limits_{i=1,2,3}|u(z)-a_i|\leq Ke^{-k(\dist(z,T_R)-CR^{\f{\al+1}{2}})^+},\ \forall  z\in B_{\f34 R}, 
    \end{equation}
    where $(a)^+=\max\{a,0\}$.
    \item Let $O_RD^1_R$ be the approximated $a_1$-$a_3$ interface at the scale $R$, and let $\theta_R$ denote the angle between $O_RD^1_R$ and the direction $\mathbf{e}_{13}$. Then 
    \begin{equation}
  \label{angle diff} 
        |\theta_R-\theta_{2R}|\leq CR^{\f{\al-1}{2}}.
    \end{equation}
\end{enumerate}
\begin{proof}
    All the proofs can be found in \cite{geng2024uniqueness}, and we therefore omit the details here. Specifically, (1) follows directly from Section 4, Proposition 5.1 and Proposition 5.3 in \cite{geng2024uniqueness}, while (2) and (3) correspond to Proposition 6.1 and Lemma 8.1, respectively.  
\end{proof}

\end{proposition}

Now we are in the position to derive the pointwise estimate of the distance of $u(z)$ to energy wells away from the $a_1$-$a_3$ sharp interface $\{(x,0):x>0\}$. 

\begin{lemma}\label{lemma:unif exp decay}
    There exist $\beta\in (\f12,1)$ and positive constants $R_0, C, K, k$ only depending on $W$ and $u$, such that for any $x\geq R_0$,
    \begin{align}
   \label{est: close to a1}         |u(x,y)-a_1|&\leq Ke^{-k(y-Cx^\beta)},\quad y\geq Cx^\beta,\\
   \label{est: close to a3}         |u(x,y)-a_3|&\leq Ke^{-k(|y|-Cx^\beta)},\quad y\leq -Cx^\beta,\\
   \label{est: grad}         |\na u(x,y)|&\leq  Ke^{-k(y-Cx^\beta)},\quad |y|\geq Cx^\beta.
    \end{align}
\end{lemma}

\begin{proof}
    Let $\beta=\f{1+\al}{2}$, where $\al$ is the parameter in Proposition \ref{prop: appr uR}. We first show that for sufficiently large $R$,
    \begin{equation}\label{est: DR}
    D_R^1\in \{(x,y): |y|\leq C_0R^\beta\}.
    \end{equation}
    for $C_0=C_0(W,u)$ independent of $R$. 

From \eqref{angle diff}, we have 
\begin{equation}\label{est: thR}
\begin{split}
    |\theta_R|&=|\theta_R-\theta_\infty|\\
    &\leq \sum\limits_{n=0}^\infty |\theta_{2^{n}R}-\theta_{2^{n+1}R}|\\
    &\leq C(\sum\limits_{n=0}^\infty 2^{n(\beta-1)})R^{\beta-1}=C_1R^{\beta-1}.
\end{split}
\end{equation}
Furthermore, let $C_2$ denote the constant $C$ in \eqref{diff interface loc}. Set
\begin{equation*}
    C_0:=100\max\{C_1,C_2\}.
\end{equation*}
We show that this choice of $C_0$ is sufficient to validate \eqref{est: DR}. Assume by contradiction there exists a sequence $R_i\ri\infty$, such that $D_{R_i}^1=(x_i,y_i)$ satisfies $|y_i| >C_0R_i^\beta$. Since $|\theta_{R_i}|\leq C_1{R_i}^{\beta-1}$, by an elementary geometry argument we obtain
\begin{equation*}
    \dist((0,0), l_i)> \f{C_0}{2}R_i^\beta,
\end{equation*}
where $l_i$ represents the straight line passing by $D_{R_i}^1$ and $O_{R_i}$. Therefore,
\begin{equation*}
    r_i:=\dist(O_{R_i}, (0,0))>\f{C_0}{2}R_i^{\beta}\geq 50C_2R_i^{\beta}.
\end{equation*}
We focus on the ball $B_{r_i}$. From \eqref{diff interface loc} we have 
\begin{equation}
\Gamma_\delta\cap B_{r_i}\subset \Gamma_\delta\cap B_{R_i} \subset \{z: \dist(z,T_{R_i})\leq C_2R_i^{\beta}\},
\end{equation}
which implies the $\Gamma_\delta\cap B_{r_i}$ is contained within an $\f{r_i}{50}$-neighborhood of a triod (i.e. $T_{R_i}$) centered on $\pa B_{r_i}$. Outside this neighborhood, $u(z)$ remains close to one of the energy wells $a_i$. It is clear that there exists a positive constant $c>0$ such that 
\begin{equation*}
    \|u_{r_i}-u_{\mathcal{P}}\|_{L^1(B_1)}>c,\quad \forall i.
\end{equation*}
As $r_i\ri\infty$, this yields a contradiction with \eqref{L1 conv}, thereby proving \eqref{est: DR}. 

Define 
\begin{equation*}
    C:=100(C_0+C_2).
\end{equation*}
We fix a sufficiently large $x$ such that $Cx^\beta<\f12 x$. If $Cx^\beta <y <x$, then $z=(x,y)\in B_{2x}$. By \eqref{est: DR} and \eqref{est: thR} we have
\begin{equation*}
    \dist(z, T_{4x})\geq C_2(4x)^\beta,
\end{equation*}
which with \eqref{exp decay: TR} shows that \eqref{est: close to a1} holds for $Cx^\beta <y <x$.

For the case $y>x$, we proceed similarly, only replacing $B_{4x}$ with $B_{4y}$, to obtain
\begin{equation*}
    |u(x,y)-a_1|\leq Ke^{-k(\dist((x,y),T_{4y})-Cy^\beta)}.
\end{equation*}
When $x$, and thus $y$, is chosen sufficiently large, we have $\dist((x,y),T_{4y})-Cy^\beta>\f13(y-Cx^\beta)$. Therefore, by updating $k$ to $\frac{k}3$, we can verify \eqref{est: close to a1} for the case $y>x$. 

\eqref{est: close to a3} can be proven in exactly the same manner as \eqref{est: close to a1}, while \eqref{est: grad} follows directly from the standard elliptic regularity theory. This completes the proof.
\end{proof}

For $r_1<r_2$, $\theta_1<\theta$, we define the angular section
\begin{equation*}
    \mathcal{A}(r_1,r_2;\theta_1,\theta_2):= \{z=(r\cos{\theta},r\sin{\theta}): r\in(r_1,r_2), \theta\in(\theta_1,\theta_2)\}.
\end{equation*}

Next we prove the following sharp energy lower and upper bounds on $B_R$.

\begin{lemma}\label{lemma: sharp bdd on BR}
    There are constants $R_1$, $C$, depending on $W,\,u$ only, such that for $R>R_1$,
    \begin{equation}\label{est: sharp ene bdd on BR}
        3\sigma R-C\leq E(u,B_R)\leq 3\sigma R+C.
    \end{equation}
\end{lemma}

\begin{proof}
    Let $R_1:=10R_0$, where $R_0$ is the constant in Lemma \ref{lemma:unif exp decay}. For any $R>R_1$, by the exponential decay results stated in Lemma \eqref{lemma:unif exp decay}, which also hold for the interfaces $\pa \mathcal{D}_1\cap \pa\mathcal{D}_2$ and $\pa \mathcal{D}_2\cap\pa\mathcal{D}_3$, we have
\begin{equation}\label{est: close to ai on 3 interface}
   \max\left\{|u(r,\!\sqrt{3}r)-a_1|, |u(r,\!-\sqrt{3}r)-a_3|, |u(-2r,\!0)-a_2|\right\} \leq Ke^{-k(\sqrt{3}r-Cr^\beta)},\  r\in [R_0,R].
\end{equation}
This, together with Lemma \ref{lemma: 1D energy estimate}, implies
\begin{equation}\label{est: sharp lower bdd on BR}
\begin{split}
    &E(u,{B_R\setminus B_{R_0}})\\
    \geq & \int_{R_0}^{R} \int_{\pa B_r} \left(\f12|\pa_T u|^2+W(u)\right)\,d\ch\,dr\\
    \geq & 3\sigma(R-R_0)-\int_{R_0}^{R}3Ce^{-2k(\sqrt{3}r-Cr^\beta)}\,dr\\
    \geq &3\sigma R-C(W,u),
\end{split}    
\end{equation}
which proves the lower bound in \eqref{est: sharp ene bdd on BR}.

Now we fix two small constants $\delta$ and $\e$ depending on $W$ and $u$. The energy limit \eqref{eq: energy limit} and Fubini's Theorem imply that for suitably large $R$, one can always find a $\bar{R}\in (R,2R)$ such that 
\begin{equation*}
    3\sigma-\e\leq \int_{\pa B_{\bar{R}}} \left(\f12|\pa_{T} u|^2+W(u)\right)\,d\ch \leq 3\sigma+\e. 
\end{equation*}
Then we invoke the arguments from \cite[Section 4]{geng2024uniqueness} to obtain the following nice behavior of $u$ on $\pa B_{\bar{R}}$, when $\delta$ and $\e$ are chosen suitably small. 
\begin{enumerate}
    \item There are three arcs $I_1, I_2, I_3\subset \pa B_{\bar{R}}$ such that
    \begin{equation*}
        z\in I_i \Rightarrow |u(z)-a_i|\leq \delta,\ \ i=1,2,3. 
    \end{equation*}
    \item There are three arcs $I_{12}, I_{23}, I_{13}$ denoting transitional arcs between $I_i's$,  such that  
    \begin{equation*}
    \begin{split}
        &C_1\leq \ch(I_{ij})\leq \f{C_2}{\delta^2}, \ \forall i\neq j.\\
        & \bigcup\limits_{i\neq j} I_{ij} =\pa B_{\bar{R}}\setminus (\bigcup\limits_{i} I_i).
        \end{split}
    \end{equation*}
    Here $C_1,C_2$ are constants that only depends on $W$.
\end{enumerate}

Now we construct an energy competitor $V(z)$ on $\pa B_{\bar{R}}$. For $z\in \pa B_{\bar{R}-1}$, we define 
\begin{equation*}
    v_1(z)=\begin{cases} 
         a_i, & z\in \f{\bar{R}-1}{\bar{R}}I_i,\ i=1,2,3\\
         \text{smooth connection of $a_i$ to $a_j$}, & z\in \f{\bar{R}-1}{\bar{R}} I_{ij}.
    \end{cases}
\end{equation*}

On the closed annulus $\{\bar{R}-1\leq |z|\leq \bar{R}\}$, we set
\begin{equation*}
    v(z)=\begin{cases}
        u(z), & |z|=\bar{R},\\
        v_1(z), & |z|=\bar{R}-1,\\
        \text{linear interpolation}, & |z|\in(\bar{R}-1,\bar{R}).
    \end{cases}
\end{equation*}
We split the energy of $v$ on the annulus into four parts:
\begin{align*}
    &E(v, B_{\bar{R}}\setminus B_{\bar{R}-1})\\
    = & \sum\limits_{j=1}^3\int_{\{\theta: \bar{R}e^{i\theta}\in I_j\}} \int_{\bar{R}-1}^{\bar{R}} \f12|\pa_r v|^2\,rdrd\theta\\
    &\quad + \sum\limits_{j=1}^3\int_{\{\theta: \bar{R}e^{i\theta}\in I_j\}} \int_{\bar{R}-1}^{\bar{R}}\f12|\pa_T v|^2\,rdrd\theta\\
    &\quad + \sum\limits_{j=1}^3\int_{\{\theta: \bar{R}e^{i\theta}\in I_j\}} \int_{\bar{R}-1}^{\bar{R}}W(v)\,rdrd\theta\\
    &\quad +\int_{\{\theta: \bar{R}e^{i\theta}\in I_{12}\cup I_{23}\cup I_{13}\}} \int_{\bar{R}-1}^{\bar{R}}\left(\f12|\na v|^2+W(v)\right)\,rdrd\theta\\
    =: & \mathcal{E}_1+\mathcal{E}_2+\mathcal{E}_3+\mathcal{E}_4.
\end{align*}
For $\mathcal{E}_4$, since the length of $I_{ij}$ is $O(1)$ and $|\na u|$ and $W(u)$ are uniformly bounded, immediately one gets
\begin{equation*}
    \mathcal{E}_4\leq C(W,\delta).
\end{equation*}
For quantities $\mathcal{E}_1, \mathcal{E}_2, \mathcal{E}_3$, we estimate using Lemma \ref{lemma: potential energy estimate}:
\begin{align*}
    \mathcal{E}_1&\leq \sum\limits_{j} \int_{\bar{R}-1}^{\bar{R}} \int_{\{\theta: \bar{R}e^{i\theta}\in I_j\}} \f12|u(\bar{R},\theta)-a_j|^2r\,drd\theta\\
    &\leq \sum\limits_{j}  \int_{I_j} |u-a_j|^2\,d\ch\\
    &\leq \sum\limits_{j}  \int_{I_j} CW(u)\,d\ch\leq C.
\end{align*}
\begin{equation*}
    \mathcal{E}_2\leq \int_{\pa B_{\bar{R}}} \f12|\pa_T u|^2\,d\ch\leq 4\sigma.
\end{equation*}
\begin{equation*}
    \mathcal{E}_3\leq \sum\limits_{j=1}^3\int_{\{\theta: \bar{R}e^{i\theta}\in I_j\}} \int_{\bar{R}-1}^{\bar{R}} C|u-a_j|^2\,drd\theta\leq C.
\end{equation*}

Combining all the estimates above implies 
\begin{equation*}
    \int_{B_{\bar{R}}\setminus B_{\bar{R}-1}} \left( \f12|\na v|^2+ W(v) \right)\leq C, \quad \text{ for some }C=C(W).
\end{equation*}
Here we can choose $\delta$ such that Lemma \ref{lemma: potential energy estimate} is applicable. 

Then we follow the same construction of the energy competitor in \cite[Appendix A]{alikakos2024triple} to complete the construction of $v(z)$ inside $B_{\bar{R}-1}$, which satisfies
\begin{equation*}
    \int_{B_{\bar{R}-1}}\left(\f12|\na v|^2+W(v)\right)\,dz\leq 3\sigma(\bar{R}-1)+C(W).
\end{equation*}

We explain the rough idea: First we pick three midpoints of three transition layers, $I_{12}, I_{23}, I_{13}$, denoted by $A_{12}, A_{23}, A_{13}$ respectively. For the approximated interface $OA_{ij}$, we define the function $v(z)=U_{ij}(\dist(z, OA_{ij}))$ within the annulus sector $\mathcal{A}(R_0,\bar{R}-2;\theta(A_{ij})-\f{\pi}{6}, \theta_{ij}+\f{\pi}{6})$, where $
\dist(\cdot,\cdot)$ represents the signed distance. Most of the energy will be concentrated in these three annulus sectors, which adds up to $3\sigma(\bar{R}-R_0)-C$. For the remaining regions of $B_{\bar{R}-1}$, we interpolate linearly, and the residual energy can be shown to be bound by a constant $C$. For the detailed proof, we refer interested readers to \cite[Appendix A]{alikakos2024triple}. 

Therefore, we combine the estimates and the minimality of $u$ to get 
\begin{equation*}
    E(u,B_{\bar{R}}) \leq E(v,B_{\bar{R}})\leq 3\sigma \bar{R}+C.
\end{equation*}
Recall that $\bar{R}\in (R,2R)$. We utilize \eqref{est: close to ai on 3 interface} again to bound the energy from below on the annulus $B_{\bar{R}}\setminus B_R$:
\begin{equation*}
    \int_{B_{\bar{R}}\setminus B_R} \left(\f12|\na u|^2+W(u)\right)\,dz\geq 3(\bar{R}-R)-C.
\end{equation*}
Combining the two inequalities above we obtain the upper bound in \eqref{est: sharp ene bdd on BR}, which completes the proof.

\end{proof}

\begin{rmk}
    A direct consequence of Lemma \ref{lemma: sharp bdd on BR} is
    \begin{equation}
        \label{est: radial ene}
        \int_{\BR^2} |\pa_r u|^2\,dz\leq C.
    \end{equation}
    This follows from the fact that only tangential deformation is considered when estimating the lower bound in \eqref{est: sharp lower bdd on BR}.  
\end{rmk}

\begin{lemma}\label{lemma: diffuse interface R1/2} 
For suitably small $\delta$, there exist $R_0$, $C_1$ and $C_2$, depending on $W$, $u$ and $\delta$, such that for $R>R_0$, there exists an angle $\theta^R\in (-C_1R^{\beta-1},C_1R^{\beta-1})$ that satisfies
\begin{equation}\label{est: diffuse interface R1/2 close}
    \left(\G_\delta \cap \mathcal{A}\left(R,9R; -\f{\pi}{3},\f{\pi}{3}\right)\right)\subset \mathcal{A}\left(R,9R; \theta^R-C_2R^{-\f12}, \theta^R+C_2R^{-\f12}\right)
\end{equation}

Moreover, there exist $K,k$ depending on $W$ and $u$, such that 
\begin{align}
   \label{R1/2 close, exp decay1} &|u(z)-a_1|\leq Ke^{-kr(\theta-\theta^R-C_2R^{-\f12})},\ z=(r\cos\theta,r\sin\theta)\in \mathcal{A}\left(2R,8R; \theta^R+C_2R^{-\f12},\f{\pi}{2}\right),\\
   \label{R1/2 close, exp decay2} &|u(z)-a_3|\leq Ke^{-kr(\theta^R-C_2R^{-\f12}-\theta)},\ z=(r\cos\theta,r\sin\theta)\in \mathcal{A}\left(2R,8R; -\f{\pi}{2}, \theta^R-C_2R^{-\f12}\right).
\end{align}
\end{lemma}

\begin{proof}

    For sufficiently large $R$, using Fubini Theorem, \eqref{eq: energy limit} and Proposition \ref{prop: appr uR}, we follow the same argument as in the proof of Lemma \ref{lemma: sharp bdd on BR} to derive the existence of $R_1\in(\f12 R,R)$, $R_2\in(9R,10R)$ that satisfies the following properties:
    \begin{enumerate}
        \item $3\sigma-\e<\f{1}{R_i}\int_{\pa B_{R_i}} \left(\f12|\pa_T u|^2+W(u)\right)<3\sigma+\e$, for $i=1,2$, $\e\ll 1$. 
        \item There are $\theta_1, \theta_2 \in (-CR^{\beta-1}, CR^{\beta-1})$ such that  
        \begin{align}
 \label{close to a1}           &|u(z)-a_1|\leq \f12\delta, \text{ for }z=(R_i\cos\theta, R_i\sin\theta),\ \theta\in (\theta_i+\f{C}{R_i},\f{\pi}{3}), \ i=1,2,\\
 \label{close to a3}           &|u(z)-a_3|\leq \f12\delta, \text{ for }z=(R_i\cos\theta, R_i\sin\theta),\ \theta\in (-\f{\pi}{3}, \theta_i-\f{C}{R_i}), \ i=1,2.
        \end{align}
    \end{enumerate}
    The inequalities above basically indicate that on $\pa B_{R_i}\cap \{\theta\in(-\f{\pi}{3},\f{\pi}{3})\}$, $u(z)$ will be uniformly close to $a_1$ or $a_3$ outside a transition layer of size $O(1)$. 

    Next we claim that there exists a constant $C_3=C_3(W,u)$ such that
    \begin{equation*}
        |\theta_1-\theta_2|\leq C_3R^{-\frac12}.
    \end{equation*}

    \textbf{Proof of the Claim: } We argue by contradiction. Assume that $\theta_1-\theta_2>C_3R^{-\f12}$ for some $C_3=C_3(W,u)$ to be determined. We can take $R$ large enough to guarantee that 
\begin{equation*}
(\theta_1-\frac{C}{R_1})-(\theta_2+\frac{C}{R_2})>\frac{C_3}{2}R^{-\f12},
\end{equation*}
where the constant $C$ appears in \eqref{close to a1} and \eqref{close to a3}. Then for $\theta\in(\theta_2+\frac{C}{R_2}, \theta_1-\frac{C}{R_1})$, we have 
\begin{equation}\label{close to a1,a3}
|u(R_1,\theta)-a_3|\leq \delta,\quad |u(R_2,\theta)-a_1|\leq \delta.
\end{equation}
Now we calculate the energy on the annulus $\mathcal{A}(R_1,R_2;0,2\pi)$. Let $\varphi\in (0,\frac{\pi}{2})$ be an angle whose value will be specified later, we compute
\begin{equation}\label{ene est: rad and tang}
\begin{split}
&E(u,\mathcal{A}(R_1,R_2;0,2\pi)) \\
\geq & \int_{R_1}^{R_2} \int_{\pa B_{r}}\left(\f12|\pa_T u|^2+\sin^2{\varphi}W(u)\right)\,d\mathcal{H}^1\,dr + \int_{\theta_2+\frac{C}{R_2}}^{\theta_1-\frac{C}{R_1}} \int_{R_1}^{R_2} \left(\f12|\pa_r u|^2+\cos^2{\varphi}W(u)\right)r\,dr\,d\theta\\
\geq & \int_{R_1}^{R_2}  \sin\varphi(3\sigma-Ce^{-kr})\,dr + \f{C_3}{2}R^{-\f12}\cdot(\cos\varphi(\sigma-C\delta^2))\cdot(\f12R)\\
\geq & \left[3\sin\varphi(R_2-R_1)+\cos\varphi\f{C_3R^{\f12}}{8}\right]\sigma-C_4(W).
\end{split}
\end{equation}
Since the estimate above applies to any $\varphi$, we take $\varphi=\arctan{\f{3(R_2-R_1)}{\f{C_3}{8}R^{\f12}}}$ to obtain
\begin{equation}\label{est: lower bdd on annulus R1 R2}
\begin{split}
&E(u,\mathcal{A}(R_1,R_2;0,2\pi)) \\
\geq & \left( 9(R_2-R_1)^2+  \f{C_3^2R}{64} \right)^{\f12}\sigma-C_4\\
\geq & \left(3(R_2-R_1)+ \f12 \f{C_3^2R}{64\cdot 3(R_2-R_1)}\right)\sigma-C_4\\
\geq & 3(R_2-R_1)\sigma +\f{C_3^2}{3840}\sigma-C_4. 
\end{split}
\end{equation}

Applying \eqref{est: sharp ene bdd on BR} to $R_1$ and $R_2$ implies that 
\begin{equation}\label{est:upper bdd on A(R1 R2)}
    E(u,\mathcal{A}(R_1,R_2;0,2\pi))  \leq 3(R_2-R_1)\sigma+C_5
\end{equation}

Now we can take $C_3(W,u)$ large enough such that 
\begin{equation*}
    \f{C_3^2\sigma}{3840}-C_4>2C_5,
\end{equation*}
which yields a contradiction. This completes the proof of the claim.

We set
\begin{equation*}
    \theta^R:=\f{\theta_1+\theta_2}{2}.
\end{equation*}
We further show that there exists a constant $C_2(W,u)$ such that one can find $\tilde{\theta}_1\in (\theta^R, \theta^R+C_2R^{-\f12})$ and $\tilde{\theta}_3\in (\theta^R-C_2R^{-\f12}, \theta^R)$ that satisfy
\begin{equation}\label{2 radial segments near theta13}
\begin{split}
    &|u(r,\tilde{\theta}_1)-a_1|\leq \delta, \quad \forall r\in(R_1,R_2),\\
    &|u(r,\tilde{\theta}_3)-a_3|\leq \delta, \quad \forall r\in(R_1,R_2).
\end{split}
\end{equation}
The proof is based on the same idea as the proof of the claim. If for all $\theta\in  (\theta^R, \theta^R+C_2R^{-\f12})$, there exists at least one point $r(\theta)\in(R_1,R_2)$ such that $|u(r(\theta),\theta)-a_1|>\d$, then for such $\theta$, there will be some non trivial energy generated along the radial segment $\{re^{i\theta}: r\in[R_1,R_2]\}$. This is due to the boundary constraint \eqref{close to a1}. Using similar estimates as \eqref{ene est: rad and tang} and \eqref{est: lower bdd on annulus R1 R2}, which split the energy into tangential and radial parts, we can obtain
\begin{equation*}
    \int_{\mathcal{A}(R_1,R_2;0,2\pi)}\left(\f12|\na u|^2+W(u)\right)\,dz\geq 3\sigma(R_2-R_1)+ K_1C_2^2\delta^2-K_2, 
\end{equation*}
where $K_1,K_2$ depend on $W,u$. Taking $C_2$ sufficiently large yields a contradiction with \eqref{est:upper bdd on A(R1 R2)}. The existence of $\tilde{\theta}_3$ follows by the same argument. 

Using \eqref{est: close to a1}, \eqref{close to a1} and \eqref{2 radial segments near theta13}, we find 
\begin{equation*}
    |u(z)-a_1|\leq \delta, \quad\forall z\in\pa \mathcal{A}(R_1,R_2; \tilde{\theta}_1,\f{\pi}{3}).
\end{equation*}
Invoking the variational maximum principle in \cite{AF2}, we obtain
\begin{equation*}
    |u(z)-a_1|\leq \delta, \quad\forall z\in\mathcal{A}(R_1,R_2; \tilde{\theta}_1,\f{\pi}{3}).
\end{equation*}
Similarly,
\begin{equation*}
    |u(z)-a_3|\leq \delta, \quad\forall z\in\mathcal{A}(R_1,R_2; -\f{\pi}{3}, \tilde{\theta}_3).
\end{equation*}
Since $\tilde{\theta}_1, \tilde{\theta}_3$ are in the $C_2R^{-\f12}$ neighborhood of $\theta^R$, \eqref{est: diffuse interface R1/2 close} follows immediately. The exponential decay in \eqref{R1/2 close, exp decay1} and \eqref{R1/2 close, exp decay2} follows from standard estimates using the comparison principle in elliptic theory, see \cite[Lemma 4.5 \& Proposition 5.2]{afs-book} for detailed arguments.  The proof of Lemma \ref{lemma: diffuse interface R1/2} is complete.
 
\end{proof}

Given $R$ large, applying \eqref{est: diffuse interface R1/2 close} to radius $R$ and $2R$ implies 
\begin{equation*}
\begin{split}
&\left(\Gamma_\delta\cap \mathcal{A}(2R,9R;-\f{\pi}{3},\f{\pi}{3})\right)\\
\subset & \bigg[\mathcal{A}\left(2R,9R; \theta^R-C_2R^{-\f12}, \theta^R+C_2R^{-\f12}\right)\cap \mathcal{A}\left(2R,9R; \theta^{2R}-C_2(2R)^{-\f12}, \theta^{2R}+C_2(2R)^{-\f12}\right)\bigg].   
\end{split}
\end{equation*}
Consequently, we get the similar estimate of the closeness for $\theta_{13}^R$ at comparable scalings as in \eqref{angle diff}, but with the improved power,
\begin{equation*}
    |\theta^R-\theta^{2R}|\leq CR^{-\f12}.
\end{equation*}
We argue as in \eqref{est: thR} to get
\begin{equation*}
    |\theta^R|\leq CR^{-\f12}.
\end{equation*}
Furthermore, following the proof of Lemma \ref{lemma:unif exp decay}, we can refine the power $\beta$ in \eqref{est: close to a1}--\eqref{est: grad} to $\f12$. Specifically, we have
\begin{align}
\label{est: close to a1, new}         |u(x,y)-a_1|&\leq Ke^{-k(y-Cx^{\f12})},\quad y\geq Cx^{\f12},\\
   \label{est: close to a3, new}         |u(x,y)-a_3|&\leq Ke^{-k(|y|-Cx^{\f12})},\quad y\leq -Cx^{\f12},\\
   \label{est: grad, new}         |\na u(x,y)|&\leq  Ke^{-k(y-Cx^{\f12})},\quad |y|\geq Cx^{\f12},
\end{align}
when $x>R_0(W,u)$. In short, Lemma \ref{lemma: diffuse interface R1/2} together with the analysis above indicate that the parameters $(\alpha,\beta)$ in Proposition \ref{prop: appr uR} and Lemma \ref{lemma:unif exp decay} can be strengthened to $(0,\f12)$. 

Considering a ray from the origin $\{re^{i\theta}: r\in(0,\infty),\theta=\theta_0\}$, we obtain the following direct consequence of the exponential decay estimates \eqref{est: close to a1, new} and \eqref{est: close to a3, new}, which apply to all three sharp interfaces. 
\begin{lemma}\label{lemma: exp conv along a ray}
  If $\theta\notin \{0,\f{2\pi}{3},\f{4\pi}{3}\}$, then there exist $K,k$ and $R_0$ such that 
  \begin{equation*}
      |u(r,\theta)-a_i|\leq Ke^{-k(r-R_0)},\ \forall r\geq R_0,
  \end{equation*}
  where $a_i$ is the phase associated with $\theta$, that is, $\{re^{i\theta}: r\in(0,\infty),\theta\text{ fixed}\}\in \mathcal{D}_i$.

  Moreover, if $\theta$ belongs to a compact set $\mathcal{K} \subset \left([0,2\pi)\setminus \{0,\f{2\pi}{3},\f{4\pi}{3}\}\right)$, then the constants $K,k,R_0$ can be chosen uniformly with respect to $\mathcal{K}$.  
\end{lemma}

The proof is omitted here, as it follows immediately from \eqref{est: close to a1, new}--\eqref{est: grad, new}.

\subsection{Small horizontal deformation}

Let $\BR_{x+}^2$ denote the half plane $\{(x,y):x>0\}$. We will derive the key estimate that $\int_{\BR_{x+}^2}|\pa_x u|^2\,dz$ is finite.

For $R$ sufficiently large, define an equilateral triangle $S_R$ by 
    \begin{equation}\label{pa SR}
    \begin{split}
        \pa S_R=&\{(x,y): x=R,\, y\in [-\sqrt{3}R,\sqrt{3}R]\} \bigcup \\
        &\quad \{(x,y):y=\f{\sqrt{3}}{3}(x+2R),\, x\in[-2R,R]\}\bigcup \\
        &\quad \{(x,y):y=\f{-\sqrt{3}}{3}(x+2R),\, x\in[-2R,R]\}.
        \end{split}
    \end{equation}
    $S_R$ is centered at the origin with a side length of $2\sqrt{3}R$. We further define 
    \begin{align*}
        &S^{13}_R:= \{z=(r\cos\theta,r\sin\theta)\in S_R: \theta\in (-\f{\pi}{3},\f{\pi}{3})\},\\
        &S^{12}_R:= \{z=(r\cos\theta,r\sin\theta)\in S_R: \theta\in (\f{\pi}{3},\pi)\},\\
        &S^{23}_R:= \{z=(r\cos\theta,r\sin\theta)\in S_R: \theta\in (\pi,\f{5\pi}{3})\},
    \end{align*}
    which form a 3-partition of $S_R$ such that $S^{ij}_{R}$ is symmetric with respect to the $a_i$-$a_j$ interface. We now carry out some energy estimates on $S_R$.

    \begin{figure}[htt]
    \centering
    \begin{tikzpicture}[thick]
    \fill[red!20] (0,0)--(2,0)--(2,3.464)--(-1,1.732)--(0,0);
    \fill[green!20] (0,0)--(-1,1.732)--(-4,0)--(-1,-1.732)--(0,0);
    \fill[blue!20] (0,0)--(2,0)--(2,-3.464)--(-1,-1.732)--(0,0);
    \draw (-4,0)--(2,3.464) node at (-0.5,0.8) {$S_R^{12}$};
    \draw (2,-3.464)--(2,3.464) node at (1.3,0) {$S_{R}^{13}$};
    \draw (-4,0)--(2,-3.464) node at (-0.5,-0.8) {$S_R^{23}$};
    \draw(0,0)--(2,3.464) node[above right] {$(R,\sqrt{3}R)$};
    \draw(0,0)--(-4,0) node[left] {$(-2R,0)$};
    \draw(0,0)--(2,-3.464) node[below right] {$(R,-\sqrt{3}R)$};
     \end{tikzpicture}
    \caption{Definition of $S_R$ and $S_{R}^{ij}$. Red : $a_1$, green: $a_2$, blue: $a_3$.}
    \label{fig: SR}

    \end{figure}
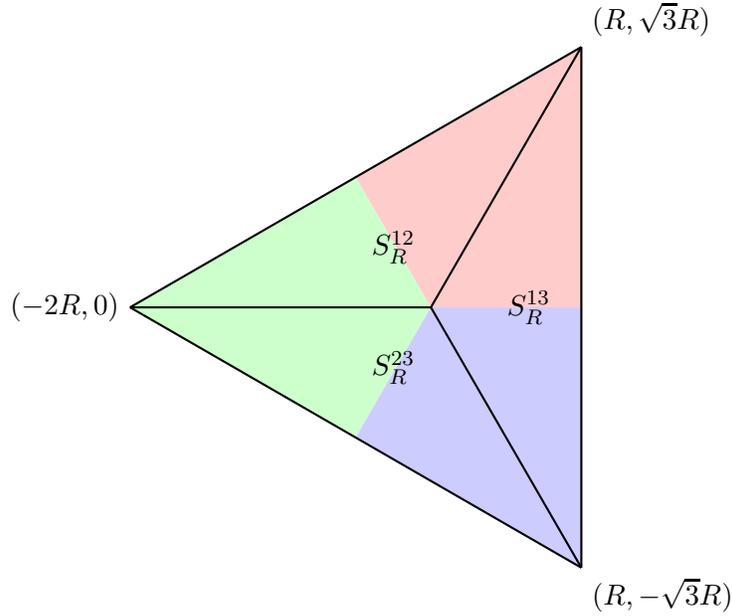
    \begin{lemma}\label{lemma: ene on SR 1}
       The following estimate holds:
       \begin{equation}\label{SR ene}
         \lim\limits_{R\ri\infty} \f{1}{R} E(u,S_R) =3\sigma.   
       \end{equation}
    \end{lemma}
\begin{proof}
  Firstly, since $B_R\subset S_R$,
   \begin{equation}\label{est: ene SR lower}
       \lim\limits_{R\ri\infty} \f{1}{R} E(u,S_R) \geq \lim\limits_{R\ri\infty} \f{1}{R} E(u,B_R)=3\sigma.
   \end{equation}
   For the inequality in the other direction, we consider a slightly larger ball $B_{R_1}$ where 
   \begin{equation*}
       R_1(R):=\sqrt{R^2+ C^2R}\sim R+\f{C^2}{2}.
   \end{equation*}
   Here $C$ is the constant in \eqref{est: close to a1, new}--\eqref{est: grad, new}. By \eqref{eq: energy limit}, we have 
   \beqo
   \lim\limits_{R\ri\infty} \f{1}{R}E(u,B_{R_1})=\lim\limits_{R\ri\infty}\f{R_1}{R}\cdot 3\sigma=3\sigma.
   \eeqo
From definition,
   \begin{equation*}
    \begin{split}
      &S_R\setminus B_{R_1}=\bigcup_{i< j} (S^{ij}_R\setminus B_{R_1}),\\
      &S^{13}_R\setminus B_{R_1}=\{(x,y): x\in [\f{R_1}{2},R), |y|\in [\sqrt{R_1^2-x^2},\sqrt{3}x)\}.
    \end{split}   
   \end{equation*}
  In particular, a point $(x,y)\in S_R^{13}\setminus B_{R_1} $ satisfies 
  \beqo
  |y|\geq CR^{\f12}\geq Cx^{\f12}. 
  \eeqo
   Utilizing estimates \eqref{est: close to a1, new}--\eqref{est: grad, new}, we compute
   \begin{align*}
   &\f1R E(u,S_{R}^{13}\setminus B_{R_1}) \\
   \leq & \f2R\left[  \int_{R_1/2}^R\left(\int_{\sqrt{R_1^2-x^2}}^{\sqrt{3}x}  C e^{-2k(y-Cx^{\f12})} \,dy\right)dx  \right]\\
   \leq & \f{C}{R}\int_{R_1/2}^{R} e^{-2k(\sqrt{R_1^2-x^2}-Cx^{\f{1}{2}})}\,dx\\
   \leq &\f{C}{R}.
   \end{align*}
   The same estimates also hold for $S_R^{23}$ and $S_{R}^{12}$. Therefore,
   \begin{equation}\label{est: ene SR upper}
   \begin{split}
      &\lim\limits_{R\ri\infty}\f{1}{R} E(u,S_R) \\
      \leq & \lim\limits_{R\ri\infty} \left(\int_{S_R\setminus B_{R_1}}+\int_{B_{R_1}}\right) \left(\f12|\na u|^2+W(u)\right)\,dz\\
      \leq & \lim\limits_{R\ri\infty} \left(\f{3C}{R}+3\sigma\right)=3\sigma.
   \end{split}  
   \end{equation}
   \eqref{est: ene SR lower} and \eqref{est: ene SR upper} together imply \eqref{SR ene}, which completes the proof.   
\end{proof} 

\begin{lemma}\label{lemma: ene bdd on SR 2}
    There exists a constant $C=C(W,u)$ such that for $R$ sufficiently large, there is a $\tilde{R}\in(R,2R)$ satisfying
    \begin{equation} \label{sharp bdd on StR}
        3\sigma\tilde{R}-C\leq \int_{S_{\tilde{R}}}  \left(\f12|\na u|^2+W(u)\right)\,dz \leq 3\sigma\tilde{R}+C.
    \end{equation}
\end{lemma}

\begin{proof}

The lower bound follows immediately by the same arguments as in \eqref{est: close to ai on 3 interface} and \eqref{est: sharp lower bdd on BR}. We are only left to prove the upper bound.

Fix a small constant $\e\ll 1$, whose value will depend on the constant $\delta$ introduced later. By Lemma \eqref{lemma: ene on SR 1} and Fubini's Theorem, for sufficiently large $R$, there exists a $\tilde{R}\in (R,2R)$ satisfying 
    \begin{equation}\label{upper bdd on pa StR}
        \int_{\pa S_{\tilde{R}}}\left(\f12 |\pa_T u |^2+W(u)\right)\,d\mathcal{H}^1\leq 3\sigma+\e.
    \end{equation}

We first observe that substituting $r=\tilde{R}$ in \eqref{est: close to ai on 3 interface} and Lemma \ref{lemma: 1D energy estimate} together yield 
\begin{equation*}
    \int_{\{\pa S_{\tilde{R}}\setminus \pa S_{\tilde{R}}^{13}\}} \left(\f12|\pa_T u|^2+W(u)\right)\,d\mathcal{H}^1\geq 2\sigma-Ce^{-k\tilde{R}}\geq 2\sigma-\e,
\end{equation*}
when $\tilde{R}$ is large enough. Combining this with \eqref{upper bdd on pa StR}, we obtain
\begin{equation}\label{est: ene vertical x=R}
    \int_{-\sqrt{3}\tilde{R}}^{\sqrt{3}\tilde{R}} \left(\f12|\pa_y u(\tilde{R},y)|^2+W(u(\tilde{R},y))\right)\,dy\leq \sigma+2\e. 
\end{equation}

This further implies that there is a $y_{\tilde{R}}$ satisfying $|y_{\tilde{R}}|\leq C\tilde{R}^{\f12}$ and 
\begin{equation}\label{loc of y_tR}
\begin{split}
    |u(\tilde{R}, y)-a_1|&\leq \delta,\quad y\in (y_{\tilde{R}}+C(W,\delta), \sqrt{3}\tilde{R}),\\
    |u(\tilde{R}, y)-a_3|&\leq \delta,\quad y\in (-\sqrt{3}\tilde{R}, y_{\tilde{R}}-C(W,\delta)),
\end{split}
\end{equation}
where $\delta\ll 1$ is a fixed small constant depending only on $W$.  Once $\delta$ is fixed, the value of $\e$ can be adjusted to ensure the existence of $y_{\tilde{R}}$. 

Next we construct an energy competitor $v$ on $S_{\tilde{R}}$ such that $v=u$ on $\pa S_{\tilde{R}}$ and $v$ satisfies the energy upper bound in \eqref{sharp bdd on StR}. Then the upper bound for $u$ follows from minimality.  

Fix a large enough $r_0$ such that $x>2Cx^{\f12}$ when $x>r_0$, where $C$ is the constant in estimates \eqref{est: close to a1, new} and \eqref{est: close to a3, new}. Additionally, assume $\tilde{R}$ is much larger than $r_0$. For the region $S^{13}_{\tilde{R}}\setminus S^{13}_{r_0}$, we first define $v(x,y)$ on the boundary: 
\begin{equation}\label{def of v on bdy StR}
    v(x,y)=\begin{cases}
    a_1, & \text{ on }\{x\in[r_0,\tilde{R}-1], y=\sqrt{3}x\},\\
    a_3, & \text{ on }\{x\in[r_0,\tilde{R}-1], y=-\sqrt{3}x\},\\
    (\tilde{R}-x)a_1+(x-(\tilde{R}-1))u(\tilde{R},\sqrt{3}\tilde{R}), &\text{ on } \{x\in (\tilde{R}-1,\tilde{R}), y=\sqrt{3}x\},\\
    (\tilde{R}-x)a_3+(x-(\tilde{R}-1))u(\tilde{R},-\sqrt{3}\tilde{R}), &\text{ on } \{x\in (\tilde{R}-1,\tilde{R}), y=-\sqrt{3}x\},\\
    u(x,y),& \text{ on }\{x=\tilde{R}, y\in[-\sqrt{3}\tilde{R},\sqrt{3}\tilde{R}]\}.
    \end{cases}
\end{equation}

Let $l_{13}$ denote the line segment connecting $(r_0,0)$ and $(\tilde{R}-1, y_{\tilde{R}})$. The slope of $l_{13}$ is $O(\tilde{R}^{-\f12})$.  We define
\begin{equation}\label{def of v in StR}
v(x,y)=\begin{cases}
  U_{13}(\dist((x,y),l_{13})), & \text{ on }S_1:=\{x\in [r_0,\tilde{R}-1], y\in [-x,x]\},\\[2.5ex]
  \begin{split}
  &\text{vertical linear interpolation between }\\[-6pt]
  &\{y= x\}\text{ and }\{y=\sqrt{3}x\},
  \end{split}
  &\text{ on }S_2:=\{x\in [r_0,\tilde{R}-1], y\in (x,\sqrt{3}x)\},\\[2.5ex]
  \begin{split}
  &\text{vertical linear interpolation between }\\[-6pt]
  &\{y=-\sqrt{3} x\}\text{ and }\{y=-x\},
  \end{split}
  &\text{ on }S_3:=\{x\in [r_0,\tilde{R}-1], y\in (-\sqrt{3}x,-x)\},\\[2.5ex]
  \begin{split}
  &\text{horizontal linear interpolation}\\[-6pt]
  &\text{between boundary data, }
  \end{split}
  &\text{ on }S_4:=\{x\in(\tilde{R}-1,\tilde{R}), y\in (-\sqrt{3}x, \sqrt{3}x)\}.
\end{cases}
\end{equation}

We compute the energy in $S_1$--$S_4$. 
\begin{equation*}
        E(v,S_1) 
        \leq  \sigma\sqrt{(\tilde{R}-1-r_0)^2+C^2\tilde{R}}\leq \sigma(\tilde{R}-r_0+C).
\end{equation*}
\begin{equation*}
    E(v,S_2\cup S_3) \leq 2\int_{r_0}^{\tilde{R}-1} \int_{x}^{\sqrt{3}x}Ce^{-ky}\,dydx\leq  C(W,u).
\end{equation*}
\begin{equation*}
  E(v,S_4) 
        \leq  \int_{S_4} \f12|\pa_x u|^2\,dz + \int_{S_4} \f12|\pa_y u|^2\,dz+ \int_{S_4} W(u)\,dz
\end{equation*}
For fixed y,  by definition we have 
\begin{equation*}
\int_{\tilde{R}-1}^{\tilde{R}} |\pa_x v(x,y)|^2\,dx=|v(\tilde{R},y)-v(\tilde{R}-1,y)|^2.    
\end{equation*}
Therefore,
\begin{equation*}
    \begin{split}
       &\int_{S_4} \f12|\pa_x v|^2\,dz\\
       \leq&  \f12\int_{-\sqrt{3}(\tilde{R}-1)}^{\sqrt{3}(\tilde{R}-1)} |v(\tilde{R},y)-v(\tilde{R}-1,y)|^2 \,dy +C\\
       \leq& C+  \int_{y_{\tilde{R}}+C}^{\sqrt{3}(\tilde{R}-1)} \left(|v(\tilde{R}-1,y)-a_1|^2 + |u(\tilde{R},y)-a_1|^2\right)\,dy\\
       & \quad + \int^{y_{\tilde{R}}-C}_{-\sqrt{3}(\tilde{R}-1)} \left(|v(\tilde{R}-1,y)-a_3|^2 + |u(\tilde{R},y)-a_3|^2\right)\,dy\\
       \leq &C+ C \int_{-\sqrt{3}(\tilde{R}-1)}^{\sqrt{3}(\tilde{R}-1)}  \left(W(v(\tilde{R}-1,y))+W(u(\tilde{R}, y))\right)\,dy\leq C
    \end{split}
\end{equation*}
where we have used \eqref{loc of y_tR} to get the second inequality.
Similarly, 
\begin{equation*}
    \begin{split}
        &\int_{S_4} W(v)\,dz\\
        \leq &C+ C \int_{\tilde{R}-1}^{\tilde{R}}\left(\int_{y_{\tilde{R}}+C}^{\sqrt{3}(\tilde{R}-1)} |v(x,y)-a_1|^2 dy+ \int^{y_{\tilde{R}}-C}_{-\sqrt{3}(\tilde{R}-1)} |v(x,y)-a_3|^2 \,dy\right)\,dx\\
        \leq &C+ C \int_{-\sqrt{3}(\tilde{R}-1)}^{\sqrt{3}(\tilde{R}-1)}  \left(W(v(\tilde{R}-1,y))+W(v(\tilde{R}, y))\right)\,dy\leq C
    \end{split}
\end{equation*}
Moreover, we have 
\begin{equation*}
    \begin{split}
        &\int_{S_4} \f12|\pa_y v|^2\,dz\\
        \leq & C+C\int_{-\sqrt{3}(\tilde{R}-1)}^{\sqrt{3}(\tilde{R}-1)}  \left( |\pa_y v(\tilde{R}-1,y)|^2+|\pa_y v(\tilde{R},y)|^2
 \right)\,dy\leq C
    \end{split}
\end{equation*}

Adding up all estimates above implies 
\begin{equation}
 \label{est: ene v S13}  E(v,S^{13}_{\tilde{R}}\setminus S_{r_0}^{13}) \leq \sigma \tilde{R} +C
\end{equation}
The same estimate also holds for $S^{12}_{\tilde{R}}\setminus S_{r_0}^{12}$ and $S^{23}_{\tilde{R}}\setminus S_{r_0}^{23}$.

Within the triangle $S_{r_0}$, the previous construction ensures that $v|_{\pa S_{r_0}}$ is Lipschitz. Hence one can extend $v$ to the interior of $S_{r_0}$ such that 
\begin{equation}\label{est: ene v Sr0}
    E(v,S_{r_0})\leq C.
\end{equation}

Combining \eqref{est: ene v S13} and \eqref{est: ene v Sr0} we get 
\begin{equation*}
    E(v,S_{\tilde{R}}) \leq 3\sigma \tilde{R}+C.
\end{equation*}
This completes the proof of Lemma \ref{lemma: ene bdd on SR 2}.

\end{proof}

Our next lemma asserts that the horizontal deformation of $u$ on $\BR^2_{x+}$ is bounded and small.

\begin{lemma}
    There is a constant $C=C(W,u)$ such that 
    \begin{equation}\label{ene bdd: px u on x+}
        \int_{\BR^2_{x+}}|\pa_x u|^2\,dxdy\leq C.
    \end{equation}
\end{lemma}

\begin{proof}
In the following, we will use the universal constants $K$, $k$, and $R_0$, which depend only on $W$ and $u$. When the constants differ across various estimates, they can be adjusted as needed without compromising the validity of the estimates.

For each $x$, define the vertical line
    \begin{equation*}
        l_x:=\{(x,y): y\in(-\infty,\infty)\}.
    \end{equation*}
 By the localization of the diffuse interface, outside the finite disk $B_{R_0}$, the diffuse interface $\G_{\d}$  is located in a small neighborhood of the limit interface $\pa\mathcal{P}$, where 
    \begin{equation*}
        \pa\mathcal{P}=\{(x,0): x>0\}\cup \{(x,-\sqrt{3}x): x<0\}\cup \{(x,\sqrt{3}x):x<0\}. 
    \end{equation*}

The size of the neighborhood is measured by $O( R^{\f12})$, i.e.
\begin{equation*}
    \dist(\G_\d\cap \pa B_{R},\  \pa\mathcal{P}\cap \pa B_{R})\leq C(W,\delta,u)R^{\f12},\ \ \forall R\geq R_0.
\end{equation*}

Thanks to the exponential decay of the distance of $u(x,y)$ to $a_i$ away from the diffuse interface, it follows that for any $x>0$, $u(x,y)$ converges to $a_1$($a_3$) exponentially as $y$ goes to $+\infty$($-\infty$). And $|\na u|$ converges exponentially to $0$ as $|y|$ goes to $+\infty$. Therefore,
\begin{equation*}
    \int_{-\infty}^\infty |\pa_x u(x,y)|^2\,dy<\infty, \ \ \forall x>0.
\end{equation*}

Consequently, to prove \eqref{ene bdd: px u on x+}, it suffices to show
\begin{equation}\label{est: px u on x>R0}
    \int_{\{x>R_0\}} |\pa_x u|^2\,dxdy<\infty,
\end{equation}
for a constant $R_0$ such that \eqref{est: close to a1, new}--\eqref{est: grad, new} hold for $x>R_0$ and Lemma \eqref{lemma: ene bdd on SR 2} applies for $R>R_0$.

Applying \eqref{est: close to a1, new} and \eqref{est: close to a3, new} to $x=R_0$ implies that 
\begin{align*}
|u(R_0,y)-a_1|&\leq Ke^{-k(y-CR_0^{\f12})},\quad y\geq CR_0^{\f12},\\
|u(R_0,y)-a_3|&\leq Ke^{-k(|y|-CR_0^{\f12})},\quad y\leq -CR_0^{\f12}.
\end{align*}
Furthermore, since $\{(x,0):x<0\}$ is the bisector of $\mathcal{D}_2$, by Lemma \ref{lemma: exp conv along a ray} we have
\begin{equation*}
     |u(x,0)-a_2|\leq Ke^{-k(|x|-R_0)}, \quad x<-R_0.
\end{equation*}

For any $r<-2R_0$, we consider the line segment $\{(x,y):x\in[r,R_0], y=\f{\sqrt{3}}{3}(x-r)\}$, which will be exponentially close to $a_2$ and $a_1$ at two endpoints, due to the estimates above. Therefore, Lemma \ref{lemma: 1D energy estimate} implies that the 1D energy on this line segment is bounded from below by $\sigma-Ce^{-k|r|}$. We have for any $R>R_0$, 
\begin{equation}\label{est: lower bdd on left upper tri in SR}
\begin{split}
    &\int_{-2R}^{R_0}\int_{0}^{\f{\sqrt{3}}{3}(x+2R)} \left(\f12|\na u |^2+W(u)\right)\,dydx\\
\geq &\f12\int_{-2R}^{-2R_0} (\sigma-Ce^{-k|r|})\,dr\geq \sigma R-C. 
\end{split}    
\end{equation}
Similarly,
\begin{equation} \label{est: lower bdd on left bottom tri in SR}
    \int_{-2R}^{R_0}\int^{0}_{-\f{\sqrt{3}}{3}(x+2R)} \left(\f12|\na u |^2+W(u)\right)\,dydx \geq \sigma R-C. 
\end{equation}

For any $R>R_0$, by Lemma \ref{lemma: ene bdd on SR 2} there exists a $\tilde{R}\in(R,2R)$ such that 
\begin{equation}\label{est: upper bdd on StR}
    E(u,S_{\tilde{R}}) \leq  3\sigma\tilde{R}+C. 
\end{equation}
This, together with \eqref{est: lower bdd on left upper tri in SR} and \eqref{est: lower bdd on left bottom tri in SR} applying to $\tilde{R}$, yields that 
\begin{equation}\label{est: upper bdd on right in SR}
    E(u,S_{\tilde{R}}\cap \{x>R_0\})\leq \sigma \tilde{R}+C. 
\end{equation}

Note that for any $x\in(R_0, \tilde{R})$, the vertical line $l_x$ intersects $\pa S_{\tilde{R}}$ at two points, $(x,\f{\sqrt{3}}{3}(x+2\tilde{R}))$ and $(x,-\f{\sqrt{3}}{3}(x+2\tilde{R}))$. \eqref{est: close to a1, new} and \eqref{est: close to a3, new} imply that 
\begin{equation*}
    |u(x,\f{\sqrt{3}}{3}(x+2\tilde{R}))-a_1|\leq Ke^{-k\tilde{R}},\ \ |u(x,\f{-\sqrt{3}}{3}(x+2\tilde{R}))-a_3|\leq Ke^{-k\tilde{R}},
\end{equation*}
which leads to 
\begin{equation*}
    \int_{-\f{\sqrt{3}}{3}(x+2\tilde{R})}^{\f{\sqrt{3}}{3}(x+2\tilde{R})} \left(\f12|\pa_y u(x,y)|^2+W(u(x,y))\right)\,dy\geq \sigma-Ce^{-2k\tilde{R}}
\end{equation*}
thanks to Lemma \ref{lemma: 1D energy estimate}. From \eqref{est: upper bdd on right in SR} we have
\begin{equation}\label{est: hor gradient on right of SR}
    \begin{split}
        &\int_{S_{\tilde{R}}\cap \{x>R_0\}} \f12|\pa_x u|^2 \,dz\\
        \leq & (\sigma \tilde{R}+C) -\int_{S_{\tilde{R}}\cap \{x>R_0\}} \left( \f12|\pa_y u|^2+W(u)\right)\,dz\\
        \leq & (\sigma \tilde{R}+C) -(\tilde{R}-R_0)(\sigma-Ce^{-2k\tilde{R}}) \leq C,
    \end{split}
\end{equation}
where the last constant $C$ depends only on $W$ and $u$. The final step is to use \eqref{est: grad, new} to get
\begin{equation*}
\lim\limits_{\tilde{R}\ri\infty}\int_{R_0}^{\tilde{R}}\int_{|y|>\f{\sqrt{3}}{3}(x+2\tilde{R})} |\pa_x u|^2\,dydx=0.    
\end{equation*}

Combining this and \eqref{est: hor gradient on right of SR} gives
\begin{equation*}
    \int_{R_0<|x|<\tilde{R}} |\pa_x u|^2\,dz\leq C,
\end{equation*}
where $C$ is independent of $\tilde{R}$. Now \eqref{est: px u on x>R0} follows immediately by letting $\tilde{R}\ri\infty$. The proof is complete.
\end{proof}

\begin{corol}\label{corol: ene in 0<x<R}
    There exists a constant $C=C(W,u)$ such that 
    \begin{equation}
        \label{est: ene on x>0}
        \int_{\{0<x<R\}} \left(\f12|\na u|^2+W(u)\right)\,dz\leq \sigma R+C, \quad \forall R>0.
    \end{equation}
\end{corol}

\begin{proof}
   Let $R_0=R_0(W,u)$ be the same constant as in \eqref{est: px u on x>R0}. By Lemma \ref{lemma: exp conv along a ray} we know that $u(x,y)$ will exponentially converge to $a_1$($a_3$) as $y$ goes to $\infty$($-\infty$) uniformly for $x\in [0,R_0]$. Therefore,
    \begin{equation}\label{est: finite energy 0<x<R0}
        \int_{\{0<x<R_0\}} \left(\f12|\na u|^2+W(u)\right)\,dz\leq C.
    \end{equation}
    Now \eqref{est: ene on x>0} follows directly from \eqref{est: finite energy 0<x<R0}, \eqref{est: upper bdd on right in SR} and the pointwise estimates \eqref{est: close to a1, new}--\eqref{est: grad, new}.     
\end{proof}

\subsection{Approximation by the heteroclinic connection}

Lemma \ref{lemma: diffuse interface R1/2} implies that the diffuse interface $\G_\d$ is confined within an $O(R^{\f12})$ neighborhood of the limiting interface. We now argue that $O(R^{\f12})$ can be improved to $O(1)$ and conclude the proof of Theorem \ref{thm: rigidity triple junction}.

Define the following notations: 

\begin{itemize}
\itemsep1em
    \item Denote by $\|v\|_s$ the $H^s(\BR)$ norm of a function $v: \BR\ri \BR^2$ for $s\geq 0$, In particular, $\|v\|_0$ denotes the $L^2$ norm.
    \item $\mathcal{U}:=\{U(\cdot-m): m\in \BR\}$ is the set of translations of the connection $U=U_{31}$. 
    \item $\mathcal{S}:=\{v\in H^1_{loc}(\BR,\BR^2): \int_{\BR}\left(\f12|v'|^2+W(v)\right)\,dx<\infty, \lim\limits_{x\ri-\infty} v(x)=a_3,\  \lim\limits_{x\ri\infty} v(x)=a_1\}$.
    \item $J(v):=\int_{\BR}\left(\f12|v'|^2+W(v)\right)\,dx$, $\forall v\in \mathcal{S}$.
    \item $d_s(v,\mathcal{U}):=\inf\{\|v-w\|_s,\ w\in \mathcal{U}\}$. 
\end{itemize}

We rely on the following results of \cite{schatzman2002asymmetric} saying that if the one dimensional energy of $v(x)$ is close to $\sigma$, then $v(x)$ can be well approximated by a translation of $U$.
\begin{proposition}\label{prop: 1D result from Schatzman}
Suppose (H1), (H2) hold. There exist positive constants $\e$ and $\alpha$ such that the following hold:

\begin{enumerate}
\itemsep1em
    \item (\cite[Lemma 2.1]{schatzman2002asymmetric}) For $s=0,1$, if $d_s(v,\mathcal{U})\leq \e$, then there is a unique $h_s(v)$ such that 
    \begin{equation*}
        d_s(v,\mathcal{U})=\|v-U(\cdot- h_s(v))\|_s.
    \end{equation*}
    Moreover, $h_s$ is a function of class $C^{3-s}$ of $v$ for $s=0,1$.
    \item (\cite[Lemma 4.5]{schatzman2002asymmetric}) If $d_1(v,\mathcal{U})\leq \e$, then
    \begin{equation}\label{ineq: energy controls H1 dist}
        J(v)-\sigma \geq \alpha \|v-U(\cdot-h_0(v))\|_1^2 \geq  \alpha  d_1(v,\mathcal{U})^2.
        \end{equation}
    \item (\cite[Corollary 4.6]{schatzman2002asymmetric}) If $d_1(v,\mathcal{U})> \e$, then $J(v)> \al\e^2$. 
    \end{enumerate}
\end{proposition}

In the rest of this paper we fix the constants $\e$ and $\al$, as defined in Proposition \ref{prop: 1D result from Schatzman}. Moreover, we may assume $\e\ll 1$, since $\alpha$ in \eqref{ineq: energy controls H1 dist} is independent of $d_1(v,\mathcal{U})$.  Define the set
\begin{equation}\label{def: large ene set}
    \mathcal{B}:= \{x>0: d_1(u(x,\cdot))\geq \e\}. 
\end{equation}
From Corollary \ref{corol: ene in 0<x<R} and \eqref{ineq: energy controls H1 dist}, it follows that 
\begin{equation}\label{est: mres B}
    |\mathcal{B}|\leq \f{C}{\al\e^2}=C(W,u).
\end{equation}
Proposition \ref{prop: 1D result from Schatzman} implies that for any $x\in \BR^+\setminus \mathcal{B}$, i.e., outside  a set of finite measure, the slicing function $u(x,\cdot)$ can be approximated closely in $\|\cdot\|_s$ norm by a translated connnection $U(\cdot-h_s(x))$, where $h_s(x)$ is a $C^2$ function of $x$ on $\BR^+\setminus \mathcal{B}$, for $s=0$ or $1$. 

We further recall the following identities for the 2D Allen-Cahn system, which will play a crucial role in our analysis.

\begin{proposition}[Lemma 8.2 in \cite{schatzman2002asymmetric}]\label{prop: hamil id}
    The following identities hold for $u$:
    \begin{align}
        \label{id: hamil 1}& \int_{-\infty}^{\infty} \left[\f12(|\pa_y u|^2-|\pa_x u|^2)+W(u(x,y))\right]\,dy\equiv\sigma, \ \ \forall x\in \BR^+,\\
        \label{id: hamil 2} &\qquad \int_{-\infty}^{\infty} (\pa_x u \pa_y u)\,dy\equiv 0, \ \ \forall x\in \BR^+.
    \end{align}
\end{proposition}

\begin{proof}
    The above identities were first derived in \cite{schatzman2002asymmetric}, and later extended to a more general setting in \cite{GUI2008904}. See also \cite[Section 3.4]{afs-book} for a more detailed discussion. Here we present a proof for completeness. 

    Define 
    \begin{equation*}
        G(x):=  \int_{-\infty}^{\infty} \left[\f12(|\pa_y u|^2-|\pa_x u|^2)+W(u(x,y))\right]\,dy,\quad x>0.
    \end{equation*}
    Thanks to the exponential decay estimates \eqref{est: close to a1, new}--\eqref{est: grad, new} as $|y|\ri\infty$, we have 
    \begin{equation*}
        \begin{split}
            \f{dG(x)}{dx}&=\int_{-\infty}^{\infty} \left(\pa_y u\pa^2_{xy}u-\pa_x u\pa^2_{xx} u+W_u(u)\pa_x u\right)\,dy\\
            &=\int_{-\infty}^{\infty} \left( \pa_y u\pa^2_{xy} u+\pa_x u\pa^2_{yy}u \right)\,dy\\
            &=\int_{-\infty}^{\infty} \pa_y(\pa_xu\pa_yu)\,dy=0.
        \end{split}
    \end{equation*}
    Here we have utilized the equation $\Delta u=W_u(u)$ and \eqref{est: grad, new}. Thus $G(x)\equiv C_1$ for all $x>0$. From \eqref{est: ene on x>0}, we deduce
    \begin{equation*}
        \sigma R+C\geq C_1R. 
    \end{equation*}
    Letting $R\ri\infty$ yields $C_1\leq \sigma$. On the other hand, since $u(x,\cdot)$ connects $a_3$ to $a_1$ for every $x$, and $\sigma$ is the minimal energy for such a connection, we also have
    \begin{equation*}
        C_1R\geq \sigma R-\int_{\{x>0\}} |\pa_x u|^2\,dz\geq \sigma R-C,
    \end{equation*}
    Sending $R\ri\infty$ implies $C_1\geq \sigma$. Therefore, $C_1=\sigma$ and identity \eqref{id: hamil 1} is proved.

    For identity \eqref{id: hamil 2}, we define
    \begin{equation*}
        H(x):=\int_{-\infty}^{\infty} (\pa_x u\pa_y u)\,dy,\quad x>0.
    \end{equation*}
    We compute
    \begin{equation*}
        \begin{split}
            \f{dH(x)}{dx}&=\int_{-\infty}^{\infty} (\pa^2_{xx} u\pa_y u+\pa_x u\pa^2_{xy} u)\,dy\\
            &=\int_{-\infty}^\infty (-\pa^2_{yy} u\pa_yu-W_u(u)\pa_y u+\pa_x(u)\pa^2_{xy} u)\,dy\\
            &=\int_{-\infty}^\infty -\pa_y(\f12|\pa_y u|^2-\f12|\pa _x u|^2+W(u))\,dy=0,
        \end{split}
    \end{equation*}
    where we use $|\na u(x,y)|\ri 0$ and $W(u(x,y))\ri 0$ as $|y|\ri\infty$ in the last step. Thus $H(x)\equiv C_2$ for all $x>0$. 

    For any $\tau>0$, \eqref{ene bdd: px u on x+} and \eqref{est: ene on x>0} implies that there is a $x_\tau>0$ such that 
    \begin{equation*}
        \int_{-\infty}^\infty |\pa_x u(x_\tau,y)|^2\,dy<\tau^2,\quad  \int_{-\infty}^\infty |\pa_y u(x_\tau, y)|^2\,dy<2\sigma.
    \end{equation*}
    which further gives
    \begin{equation*}
        |C_2|=|H(x_\tau)|\leq \left(\int_{-\infty}^\infty |\pa_x u|^2\,dy\right)^{\f12}\left(\int_{-\infty}^\infty |\pa_y u|^2\,dy\right)^{\f12}\leq \sqrt{2\sigma}\tau.
    \end{equation*}
    Since $\tau$ can be arbitrarily small, $C_2=0$. This proves \eqref{id: hamil 2}.
\end{proof}

From  Proposition \ref{prop: 1D result from Schatzman}, there exists a function $h_0(x)\in C^2(\BR^+\setminus \mathcal{B},\BR)$ such that 
\begin{equation*}
    d_0(u(x,\cdot),\mathcal{U})=\|u(x,\cdot)-U(\cdot-h_0(x))\|_{0}, \quad x\in \BR^+\setminus \mathcal{B}.
\end{equation*}
For simplicity in the following analysis, we will omit the subscript and denote it as $h(x)$. We have the following identities.

    \begin{equation}
        \label{id: orthogonality}
        \int_{-\infty}^\infty (u(x,y)-U(y-h(x)))\cdot U'(y-h(x)) \,dy=0,\quad  x\in \BR^+\setminus \mathcal{B},
    \end{equation}
    \begin{equation}
        \label{id: formula h'(x)}
        h'(x)=\f{\int_{-\infty}^\infty \pa_xu(x,y)\cdot U'(y-h(x))\,dy}{\int_{-\infty}^\infty\big[ |U'(y-h(x))|^2 + U''(y-h(x))(u(x,y)-U(y-h(x)))  \big]  \,dy}.
    \end{equation}

Identity \eqref{id: orthogonality} follows directly from 
\begin{equation*}
    \f{d}{d\delta} \| u(x,\cdot)-U(\cdot-h(x)+\d) \|^2_{0}=0,\quad \text{at }\d=0.
\end{equation*}
since $U(y-h(x))$ minimizes the $L^2$ distance to $u(x,y)$. And \eqref{id: formula h'(x)} follows from differentiating \eqref{id: orthogonality} with respect to $x$. 

For the numerator of \eqref{id: formula h'(x)}, we use \eqref{id: hamil 2} to deduce
\begin{equation}\label{est: h' numerate 1}
    \begin{split}
        &\left|\int_{-\infty}^\infty \pa_xu(x,y)\cdot U'(y-h(x))\,dy\right|\\
        = & \left|\int_{-\infty}^\infty \pa_xu(x,y)\cdot (U'(y-h(x))-\pa_yu(x,y))\,dy\right|\\
        \leq & \|\pa_x u(x,\cdot)\|_{0}\cdot\|u(x,\cdot)-U(\cdot-h(x))\|_1\\
        \leq & \f12\|\pa_x u(x,\cdot)\|_{0}^2+\f12\|u(x,\cdot)-U(\cdot-h(x))\|_1^2.
    \end{split}
\end{equation}

In particular, for $x\in \BR\setminus \mathcal{B}$, we deduce from \eqref{ineq: energy controls H1 dist} and \eqref{id: hamil 1} that
\begin{equation}\label{est: h' numerate 2}
    \begin{split}
    &\left|\int_{-\infty}^\infty \pa_xu(x,y)\cdot U'(y-h(x))\,dy\right|\\
    \leq &\f12\|\pa_x u(x,\cdot)\|_{0}^2+\f{1}{2\alpha}(J(u(x,\cdot))-\sigma)\\
    = &\f{\alpha+1}{2\alpha}  \|\pa_x u(x,\cdot)\|_{0}^2
    \end{split}
\end{equation}

We establish the following lemma.

\begin{lemma}\label{lemma: d0 converge to 0}
  $\lim\limits_{x\ri\infty} d_0(u(x,\cdot),\mathcal{U})=0$.
\end{lemma}

\begin{proof}
    We argue by contradiction. Suppose there exist $\e_0>0$ and a sequence $x_n\ri\infty$ such that $d_0(u(x_n,\cdot),\mathcal{U})\geq \e_0$ for each $n$. Without loss of generality, we may assume $\e_0<\e$.  

    When $d_0(u(x, \cdot),\mathcal{U})>\f{\e_0}{10}$, Proposition \ref{prop: 1D result from Schatzman} implies that $J(u(x, \cdot))>\sigma+\al(\f{\e_0}{\alpha})^2$, and therefore the set $\{x: d_0(u(x, \cdot),\mathcal{U})>\f{\e_0}{10}\}$ has finite measure. For a sufficiently large $R$, we can always find $x_R>R$ such that 
    \begin{equation*}
        d_0(u(x_R, \cdot),\mathcal{U})\leq \f{\e_0}{10}. 
    \end{equation*}
    We define 
    \begin{equation*}
    \bar{x}_R:= \inf\{x: x>x_R,\, d_0(u(x_R, \cdot)=\e_0\},
    \end{equation*}
    where the existence of $\bar{x}_R>x_R$ is guaranteed by the fact that $x_n\ri\infty$. For any $x\in (x_R,\bar{x}_R)$, we have
    \begin{equation}\label{est: derivative of d0^2}
    \begin{split}
        &\f{d}{dx} |d_0(u(x, \cdot),\mathcal{U})|^2\\
        =&\f{d}{dx} \int_{-\infty}^\infty |u(x,y)-U(y-h(x))|^2\,dy\\
        =& \int_{-\infty}^\infty \left(\pa_xu(x,y)-U'(y-h(x))\pa_xh(x)\right)\cdot\left(u(x,y)-U(y-h(x))\right)\,dy\\
        =& \int_{-\infty}^\infty \pa_xu(x,y)\cdot\left(u(x,y)-U(y-h(x))\right)\,dy\\
        \leq & \f12\left[\|\pa_x u(x,\cdot)\|_0^2 + |d_0(u(x,\cdot), \mathcal{U})|^2\right]
    \end{split}
    \end{equation}
    where we utilize \eqref{id: orthogonality} from the third line to the fourth line. Moreover, if $x\in \BR\setminus \mathcal{B}$, 
    $$
    d_0(u(x,\cdot),\mathcal{U})^2\leq \f{1}{\al}  \|\pa_x u(x,\cdot)\|_0^2
    $$ thanks to \eqref{ineq: energy controls H1 dist}. Next we compute 
    \begin{equation}\label{est: int from xR to bar xR}
        \begin{split}
            \f{99\e_0^2}{100}= &\int_{x_R}^{\bar{x}_R} \left(\f{d}{dx} |d_0(u(x, \cdot),\mathcal{U})|^2\right)\,dx\\
            \leq &\int_{(x_R,\bar{x}_R)\setminus \mathcal{B}} \f{\al+1}{2\al} \|\pa_x u(x,\cdot)\|_0^2\,dx + \int_{(x_R,\bar{x}_R)\cap \mathcal{B}} \f12\left( \|\pa_x u(x,\cdot)\|_0^2+ d_0(u(x,\cdot), \mathcal{U})^2\right)\,dx \\
            \leq &C \int_{x_R}^{\bar{x}_R} \|\pa_x u\|_0^2\,dx+ \e_0^2|(x_R,\bar{x}_R)\cap \mathcal{B}| 
        \end{split}
    \end{equation}
    Since both $\int_{x>0}|\pa_x u|^2\,dz$ and $|\mathcal{B}|$ are bounded, the last line in \eqref{est: int from xR to bar xR} tends to $0$ as $R\ri\infty$, which yields a contradiction. The proof is complete.
\end{proof}

From Lemma \ref{lemma: d0 converge to 0} and Proposition \ref{prop: 1D result from Schatzman}, there exists $R_0>0$ such that $h(x)$ is a well-defined $C^3$ function on $(R_0,\infty)$ and satisfies \eqref{id: orthogonality}, \eqref{id: formula h'(x)} and $d_0(u(x,\cdot), \mathcal{U})<\e$.

When $x>R_0$, we can estimate the denominator in \eqref{id: formula h'(x)} by
\begin{equation}
    \label{est: h' denominator}
    \begin{split}
        &\int_{-\infty}^\infty\big[ |U'(y-h(x))|^2 + U''(y-h(x))(u(x,y)-U(y-h(x)))\big]\,dy \\
        \geq & \|U'\|_0^2-\e \|U''\|_0\\
        \geq & \f12\|U'\|_0^2\geq C,
    \end{split}
\end{equation}
where we use the fact that $\|U''\|_0$ is uniformly bounded. By taking $\e\ll 1$, we can ensure that $\e \|U''\|_0\leq \f12\|U'\|_0^2$.

Combining \eqref{id: formula h'(x)}, \eqref{est: h' numerate 1}, \eqref{est: h' numerate 2} and \eqref{est: h' denominator}, we obtain
\begin{equation}\label{est: h'}
    |h'(x)|\leq \begin{cases}
    C\|\pa_x u(x,\cdot)\|_0^2, & \quad x\in \BR\setminus \mathcal{B},\\
    C\left[\|\pa_x u(x,\cdot)\|_{0}^2+\|u(x,\cdot)-U(\cdot-h(x))\|_1^2\right], & \quad x\in\mathcal{B}.
    \end{cases}
\end{equation}
where the constant $C$ only depends on $W$.

 We have the following result on the convergence of $h(x)$.

\begin{lemma}\label{lemma: h conv}
    There exists a constant $h_0$ such that $\lim\limits_{x\ri+\infty} h(x)=h_0$.
\end{lemma}
\begin{proof}
    It suffices to show that 
    \begin{equation*}
    \int_{R_0}^\infty |h'(x)|\,dx<\infty.
    \end{equation*}
    Using \eqref{est: h'}, \eqref{id: hamil 1}, \eqref{ene bdd: px u on x+} and \eqref{est: mres B}, we deduce that 
    \begin{equation*}
        \begin{split}
          &\int_{R_0}^\infty |h'(x)|\,dx\\
          \leq & C\int_{R_0}^\infty \|\pa_x u(x,\cdot)\|_0^2\,dx +\int_{\mathcal{B}} \|u(x,\cdot)-U(\cdot-h(x))\|_1^2\,dx\\
          \leq & C\int_{\{x>R_0\}}|\pa_x u|^2\,dx +C\int_{\mathcal{B}}\left( \|\pa_y u(x,\cdot)\|_0^2+\|U'\|_0^2 +\|u(x,\cdot)-U(\cdot-h(x))\|_0^2   \right)\,dx\\
          \leq & C\left(\int_{\{x>R_0\}}|\pa_x u|^2\,dx+ \int_{\mathcal{B}} J(u(x,\cdot))\,dx+\|U'\|_0^2+\e|\mathcal{B}|\right)\\
          \leq & C\left(\int_{\{x>R_0\}}|\pa_x u|^2\,dx+\int_{\mathcal{B}} |\pa_x u|^2\,dx +\sigma |\mathcal{B}| +\|U'\|_0^2+ \e|\mathcal{B}|\right)\\
          \leq & C(W,u).
        \end{split}
    \end{equation*}
    The proof is complete.
\end{proof}

Lemma \ref{lemma: d0 converge to 0} and Lemma \ref{lemma: h conv} together imply that 
\begin{equation*}
    \lim\limits_{x\ri+\infty}\|u(x,\cdot)-U(y-h_0)\|_0^2=0.
\end{equation*}
Since $u(x,\cdot)-U(y-h_0)$ are uniformly bounded in $C^{2,\alpha}(\BR,\BR^2)$ for any $\al\in(0,1)$ due to \eqref{reg of u}, the $L^2$ convergence above implies convergence in $C^{2,\alpha}(\BR,\BR^2)$, which establishes \eqref{main result: conv} for the $a_1$-$a_3$ interface. The same arguments also hold for the other two sharp interfaces. 

To complete the proof of Theorem \ref{thm: rigidity triple junction}, it remains to prove that the solution $u$ satisfies the pointwise estimate \eqref{main result: decay}. To establish this we first observe that \eqref{main result: conv} implies the diffused interface defined in \eqref{def: diffuse interface} must be contained in an $R_0$-neighborhood of the sharp interface $\pa\mathcal{P}$. Particularly, we have
\beqo
\max\limits_{x\in B(y,\f12\dist(x,\pa\mathcal{P}))}\left(\min\limits_i |u(x)-a_i|\right)< \delta, 
\eeqo
for all $y$ satisfying $\dist(y,\pa\mathcal{P})>2R_0.$ Equation \eqref{main result: decay} then follows from the comparison principle for elliptic equations. This concludes the proof of Theorem \ref{thm: rigidity triple junction}.

\section*{Acknowledgment}

The author would like to thank Nicholas D. Alikakos and Changyou Wang for their stimulating discussions and valuable suggestions. The research of Z. Geng was partially supported by AMS-Simons Travel Grant 25014588.

\bibliographystyle{acm}
\bibliography{vector-de-giorgi}

\end{document}